\documentclass[12p]{amsart}
\usepackage{amscd,amsmath,amssymb,amsthm,amsfonts,epsfig,graphics}
\usepackage{geometry}
\usepackage{graphicx}
\usepackage{mathrsfs,amssymb}
\usepackage{bm}
\usepackage{subfigure}
\usepackage{color}
\usepackage{amsmath}

\theoremstyle{plain}

\newtheorem{corollary}{Corollary}

\newtheorem{example}{Example}

\newtheorem{lemma}{Lemma}

\newtheorem{prop}{Proposition}
\newtheorem{remark}{Remark}

\newtheorem{theorem}{Theorem}
\numberwithin{equation}{section}
\newcommand{\bth}{\begin{corrolary}}
	\newcommand{\ble}{\begin{lemma}}
		\newcommand{\bcor}{\begin{corr}}

			\newcommand{\bdeff}{\begin{deff}}

				\newcommand{\bprop}{\begin{proposition}}
					\newcommand{\ele}{\end{lemma}}
				\newcommand{\ecor}{\end{corr}}
			\newcommand{\edeff}{\end{deff}}
		
		\newcommand{\eprop}{\end{proposition}}

	\newcommand{\Rn}{{\mathbb R}^n}
	
	\newcommand{\la}{\lambda}
	
	\newcommand{\eps}{\varepsilon}

	\newcommand{\supp}{\text{supp }}
	\renewcommand{\Pi}{\varPi}

	\renewcommand{\epsilon}{\varepsilon}

	\newcommand{\dist}{{{\rm dist}}}

	\newcommand{\R}{{\mathbb R}}
	
	\newcommand{\ls}{\lesssim}
	\newcommand{\M}{M}

	\newcommand{\1}{{\rm 1\hspace*{-0.4ex}%
			\rule{0.1ex}{1.52ex}\hspace*{0.2ex}}}
	\newcommand{\gs}{\gtrsim}

\begin{document}
		\title[]{A few sharp estimates of harmonic functions with applications to Steklov eigenfunctions}
		
		\author{Xing Wang and Cheng Zhang}
		
		\address{Department of Mathematics, Hunan University, Changsha, HN 410012, China}
		\email{xingwang@hnu.edu.cn}
			\address{Mathematical Sciences Center\\
			Tsinghua University\\
			Beijing, BJ 100084, China}
		\email{czhang98@tsinghua.edu.cn}

		%\subjclass[2010]{Primary ; Secondary }
		
		\keywords{}

		\dedicatory{}

		\begin{abstract}
		On smooth compact manifolds with smooth boundary, we first establish the sharp lower bounds for the restrictions of harmonic functions in terms of their frequency functions, by  using a combination of microlocal analysis and frequency function techniques by  Almgren \cite{A79} and Garofalo-Lin \cite{GL1986}. The lower bounds can be saturated by Steklov eigenfunctions on Euclidean balls and a family of symmetric warped product manifolds. Moreover, as in Sogge \cite{SoggeFIO} and Taylor \cite{TaylorPDE2}, we analyze the interior behavior of harmonic functions by constructing a parametrix for the Poisson integral operator and calculate its composition with the spectral cluster. By using microlocal analysis, we obtain several sharp estimates for the harmonic functions whose traces are quasimodes on the boundary. As applications, we establish the almost-orthogonality, bilinear estimates and transversal restriction estimates for Steklov eigenfunctions, and discuss the numerical approximation of harmonic functions.
		%the implications of our results for numerical approximation schemes of solutions to Laplace equations with various boundary conditions.
		\end{abstract}
		
		\maketitle
		
		\section{Introduction}
		
		Let $(\Omega,h)$ be a connected smooth compact   manifold with smooth boundary $(M,g)$, where $\dim\Omega=n+1\ge2$ and $h|_M=g$. The Laplace operators on $\Omega$ and $M$ are denoted by $\Delta$ and $\Delta_g$ respectively. Let $u$ be a harmonic function on $\Omega$ with boundary datum $f$, namely
		\begin{equation}\label{lapdirichlet}
			\begin{cases}
				\Delta u=0 \ \ \text{on}\ \Omega,\\
				u= f\ \ \ \ \text{on}\  \partial\Omega=M.
			\end{cases}
		\end{equation}
		Let $\mathcal{H}$ be the harmonic extension operator (or Poisson integral operator) on $\Omega$. Then $u=\mathcal{H}f$. Let $P\in OPS^1(M)$ be classical and self-adjoint with principal symbol $p(x,\xi)=|\xi|_{g(x)}$ on the boundary. Two important examples of  $P$ are $\sqrt{-\Delta_g}$ and the Dirichlet-to-Neumann (DtN) operator $\mathcal{N}$. If $e_\la$ is an eigenfunction of $\mathcal{N}$ associated with the eigenvalue $\la$, then its harmonic extension $u_\la=\mathcal{H}e_\la$ is called a Steklov eigenfunction. The Steklov eigenvalue problem has seen a surge
		of interest in the past few decades, and there are many open problems in this rapidly expanding area. For the historical background and recent developments, we recommend the survey papers by Girouard-Polterovich \cite{GP17} and Colbois-Girouard-Gordon-Sher \cite{CGGS24}.

		In this paper, we mainly  use microlocal analysis to investigate the interior behavior of $u$ with respect to the frequency of $f$. Specifically, we shall use two types of frequencies. The first type is the frequency function as in the seminal works of Almgren \cite{A79} and Garofalo-Lin \cite{GL1986}. The second type is the spectral frequency described by the associated eigenvalues in its spectral decomposition with respect to $P$. These two types of frequencies coincide when the boundary datum $f$ is an eigenfunction of the Dirichlet-to-Neumann operator  on $M$. We shall analyze them by using the microlocal analysis techniques in Sogge \cite{SoggeFIO} and Taylor \cite{TaylorPDE2}.
		
		\subsection{Geometric quantities}
		First, we introduce the geometric quantities that will be used to state our results. Let $d_h(\cdot,\cdot)$, $d_g(\cdot,\cdot)$ be the Riemannian distance functions on $\Omega$ and $M$ respectively. Let \begin{equation}
			\Omega_t = \{\omega\in \Omega: d_h(\omega,\partial \Omega)\ge t\},\ \  \Omega_t^{\mathsf{c}} = \Omega\setminus\Omega_t,\ \ \Sigma_t = \partial\Omega_t.
		\end{equation} Then $t$ is the distance from $\Sigma_t$ to the boundary.
		
		For each $x \in M$, we have a unique inward geodesic $\gamma_x(t)\subset\Omega$ normal to $M$ such that $\gamma_x(0) = x$, $|\nabla \gamma_x(t)|=1$. Let $\phi(t,x) = \gamma_x(t)$. 
		Let
		\begin{equation}\label{definj}
			\delta= \sup \{\tau>0 :\ \phi:[0,\tau]\times M \rightarrow\Omega \text{ is injective}\}.
		\end{equation}
		Then for any $t < \delta$, $\Sigma_t$ is a smooth hypersurface, and $\gamma_x(t)$ is normal to $\Sigma_t$ for any $x\in M$.

			By \eqref{definj}, we get a bijection $\phi:[0,\delta)\times M \rightarrow \Omega_{\delta}^{\mathsf{c}}$, so we can find a local coordinates system near the boundary adapted to the geodesics normal to boundary. Let $\phi^{-1} = (g_1,g_2)$ with $g_1:\Omega_{\delta}^{\mathsf{c}} \rightarrow M$ and $g_2:\Omega_{\delta}^{\mathsf{c}} \rightarrow [0,\delta)$. We choose a finite covering $M=\cup U_\nu$ and the corresponding coordinate maps $\kappa_\nu:U_\nu\to \tilde U_\nu\subset \mathbb{R}^n$. Let $$\tilde\kappa_\nu:\Omega_\nu = \phi([0,\delta)\times U_\nu)\rightarrow\ \mathbb{R}^{n+1},\ \ p\mapsto (\kappa_\nu(g_2(p)),g_1(p)).$$
		Under this setting, let $x=(x_1,x_2,...,x_n)$ be local coordinates on some open set of $M$. Then for $t\in[0,\delta)$ and fixed $x$, $(t,x_1,x_2,...,x_n)$ will be a geodesic curve starting from $x$ and normal to both $M$ and $\Sigma_t$, so $\partial_t$ is the inward normal derivative of $\Sigma_t$. Specifically, the metric tensor of $\Omega$ can be written as
		\begin{equation}\label{metric}
			h=dt^2+g_{ij}(t,x)dx_idx_j = dt^2 + g_t,
		\end{equation}
		where $g_t$ is the metric on $\Sigma_t$, $g_0=g$. Then the Laplace operator on $\Omega$ can be written as
		\begin{equation}\label{lap}
			\Delta = \partial_t^2+a(t,x)\partial_t+\Delta_{g_t}
		\end{equation}
		where
		\begin{equation}\label{aformula}
			a(t,x) = \frac{1}{2}\partial_t\det(g_{ij}(t,x))/\det(g_{ij}(t,x))
		\end{equation}
		and $\Delta_{g_t}$ is the Laplace operator on $\Sigma_t$.
		
		For each $t\in[0,\delta)$ and $x\in\Sigma_t$, we denote the Weingarten map with respect to the inward normal  unit vectors of $\Sigma_t$ as
		\begin{equation}
			\mathcal{W}_{t,x}: T_x\Sigma_t \rightarrow T_x\Sigma_t,
		\end{equation}
		and  $\mathcal{W}_{t,x}^*: T^*_x\Sigma_t \rightarrow T^*_x\Sigma_t$  is its dual. Let $k_1(x)\le ... \le k_n(x)$ be the principal curvatures at $x$, and 
		\begin{equation}
			\tilde{k}_t(x) = - \text{Tr} \mathcal{W}_{t,x}+\sup_{\xi\neq0}\frac{\langle \mathcal{W}^*_{t,x}\xi,\xi\rangle_{t}}{\langle \xi,\xi\rangle_{t} } = -\sum_{j=1}^{n-1} k_j(x),
		\end{equation}
		where the inner product $\langle \cdot,\cdot\rangle_{t}$ on $T_x^*\Sigma_t$  is defined as
		\[\langle \xi,\eta\rangle_{t} = \sum_{i,j}g^{ij}(t,x)\xi_i\eta_j.\] Let $|\xi|_{g_t(x)}=\sqrt{\langle \xi,\xi\rangle_{t}}$. Then we define $\Theta:[0,\delta)\rightarrow \mathbb{R}$ as
		\begin{equation}
			\Theta(t) = \sup_{x\in \Sigma_t}\tilde{k}_t(x) + \sup_{x\in \Sigma_t}\text{Tr} \mathcal{W}_{t,x},
		\end{equation}
		and let
		\begin{equation}\label{Kt}
			K(t) = \int_{0}^{t} e^{\int_{0}^{s}\Theta(\tau)d\tau}ds.
		\end{equation}
		By Taylor's expansion, we have
		\[K(t) = t + \frac{\Theta(0)}{2}t^2 + O(t^3).\]
		Let 
		\begin{equation}\label{rg0}r(t)=\inf_{x\in M,\ \xi\ne0}\frac{|\xi|_{g_t(x)}}{|\xi|_{g(x)}},\ \quad G(t)=\int_0^tr(s)ds.\end{equation}
		We have $r(t)=1+O(t)$ and $G(t)=t+O(t^2)$. We will see that $G(t)=K(t)$ when $\Omega$ is a Euclidean ball  or the warped product manifold in Section 3.

		\subsection{Lower bounds}We first prove the lower bounds for the $L^2$ restrictions of harmonic functions.
		\begin{theorem}\label{thm1}
			Fix any $0<\delta_0 < \delta$. There is a constant $C>0$ dependent on $(\Omega,h)$ and $\delta_0$, such that for any $0 \le t \le \delta_0$, and any $f\in L^2(M)$, we have
			\begin{equation}\label{lowestimate}
				\|\mathcal{H}f\|_{L^2(\Sigma_t)} \ge Ce^{-\Lambda K(t)}\|f\|_{L^2(M)},
			\end{equation}
			where the frequency  $\Lambda = {\int_\Omega |\nabla u|^2}/{\int_{M}u^2}$ with $u=\mathcal{H}f$, and $K(t)=t+O(t^2)$ is defined in \eqref{Kt}. 
		\end{theorem}
		 The lower bounds are sharp in the sense that they can be saturated by Steklov eigenfunctions on Euclidean balls and a family of symmetric warped product manifolds, see Section 3.  When $f\in L^2(M)\setminus H^{\frac{1}{2}}(M)$, we have $\Lambda=\infty$ and \eqref{lowestimate} is trivial   in this case. Moreover, a simple consequence is  the lower bound for the $L^p$ norm with $p>2$ 
		\begin{equation}\label{lowlp}
			\|\mathcal{H}f\|_{L^p(\Sigma_t)} \ge Ce^{-\Lambda K(t)}\|f\|_{L^2(M)}.
		\end{equation}
		It is also sharp, since it can be saturated by Example \ref{exTorus} in Section 3.

		As an application, using $\int_\Omega |\nabla u|^2 = \int_M u\partial_\nu u$, we get the following exponential lower bound for the Steklov eigenfunctions  on smooth manifolds. As before,  let $e_\la$ be a DtN eigenfunction and $u_\la=\mathcal{H}e_\la$ be the Steklov eigenfunction.
		
		\begin{corollary}\label{lowbound}
			Fix any $0<\delta_0 < \delta$. There is a constant $C>0$ dependent on $(\Omega,h)$ and $\delta_0$, such that for any $0 \le t \le \delta_0$, we have
			\begin{equation}\label{lowestimate1}
				\|u_\la\|_{L^2(\Sigma_t)} \ge Ce^{-\la K(t)}\|e_\la\|_{L^2(M)}.
			\end{equation}
		\end{corollary}

		For real analytic $\Omega$ and $M$, Galkowski-Toth \cite{GT2021} have proved some lower bounds of this type. Indeed, it follows from  \cite[Theorem 1]{GT2021} that for any $\eps>0$, $\la>\la_0(\eps)$ and $t\le \eps_0$, there is $c>0$ such that
\begin{equation}\label{gtlow}
	\|u_\la\|_{L^2(\Sigma_t)}\gs e^{-\eps \la}e^{-(t+ct^2)\la}\|e_\la\|_{L^2(M)}.
\end{equation}
		In comparison, our estimate \eqref{lowestimate1} is sharp and has no $\epsilon$-loss. Moreover, we do not need the analyticity assumptions.

	%	there exist $\eps_0, C_1>0$, for any $\eps>0$, there exist $C,\la_0>0$, such that for any $\la>\la_0$ and $0<t<\eps_0$,
		%	\[\|u_\la\|_{L^2(\Sigma_t)}+\la^{-1}\|\nabla u_\la\|_{L^2(\Sigma_t)} \ge Ce^{-\eps\la}e^{-(t+C_1 t^2)\la}\|e_\la\|_{L^2(M)}.\]

	\subsection{Upper bounds}From the examples of Steklov eigenfunctions (see Section 3), we naturally expect that any harmonic function  should rapidly decay  with respect to $\la t$, where  $t$ is the distance to the boundary and $\la$ is its ``lowest spectral frequency''. We shall give a rigorous justification for this observation.
	
Let $P\in OPS^1(M)$ be classical and self-adjoint with principal symbol $p(x,\xi)=|\xi|_{g(x)}$.  	Let $\1_{\ge \la}$ be the indicator function of $[\la,\infty)$. 
	We shall prove the following interior decay property when the spectral frequency of the boundary datum $f$ is bounded below by $\la$.
		
		\begin{theorem}\label{thm2}
			Let $0 \le \delta_0 < \delta$, $0<c<1$, $N>0$ and $\la>0$. Suppose that	$f=\1_{\ge \la}(P)f$. Then for $1<p\le\infty$ and $0\le t\le \delta_0$, we have
			\begin{equation}\label{smoothup2}
				\|\mathcal{H}f\|_{L^p(\Sigma_t)} \le C_1 e^{-c\la G(t)}\|f\|_{L^p(M)} + C_2\la^{-N}\|f\|_{L^1(M)},
			\end{equation}
			where $C_1 =C_1(\delta_0,c,p),C_2=C_2(\delta_0,c,N)$, and $G(t)=t+O(t^2)$ is defined in \eqref{rg0}.
		\end{theorem}
		The exponential term is essentially sharp, in the sense that $G(t)=K(t)$ when $\Omega$ is a Euclidean ball  or the warped product manifold in Section 3. The remainder term $O(\la^{-N})$ comes from the smoothing operators that are natural in the theory of pseudo-differential operators on smooth manifolds. The crux of the proof is the parametrix for the Poisson integral operator in Lemma \ref{parametrix}. From the proof, we shall notice that \eqref{smoothup2} is valid for all $1\le p\le \infty$ if the boundary datum $f$ has spectral frequency $\approx \la$.  Moreover, the proof also implies the derivative estimates. Indeed, if $f=\1_{\ge \la}(P)f$ then we have for $1<p<\infty$
		\begin{equation}\label{derup}
				\|\partial_{t,x}^\alpha\mathcal{H}f\|_{L^p(\Sigma_t)} \ls\la^{|\alpha|}e^{-c\la G(t)}\|f\|_{L^p(M)} + \la^{-N}\|f\|_{L^1(M)},\ \forall \alpha,\forall N.
		\end{equation}
		\begin{remark}
			In comparison to Theorem \ref{thm1}, we  remark that the upper bound of the form with $C,c>0$
			\begin{equation}\label{uppestimate}
				\|\mathcal{H}f\|_{L^2(\Sigma_t)} \le Ce^{-c\Lambda t}\|f\|_{L^2(M)}
			\end{equation}
			cannot hold for general boundary datum $f$, where the frequency  $\Lambda = {\int_\Omega |\nabla u|^2}/{\int_{M}u^2}$ with $u=\mathcal{H}f$. For example, let $f=1+e_\la$ and $u=\mathcal{H}f=1+u_\la$, where $e_\la$ is an $L^2$ normalized DtN eigenfunction. Then both  $u$ and $u_\la$ have frequencies $\Lambda \approx \la$. However, when $\la t\gg 1$, we notice that $\|u\|_{L^2(\Sigma_t)}\approx 1$ shows no decay while $\|u_\la\|_{L^2(\Sigma_t)}\ls (1+\la t)^{-N}$ decays rapidly. This key observation suggests that the interior decay property of harmonic functions should only rely on the ``lowest spectral frequency'', as described in Theorem \ref{thm2}.
		\end{remark}
		Recall Sogge's seminal work \cite{sogge88} on $L^p$ estimates. If $\1_{(\la-1,\la]}$ is the indicator function of  $(\la-1,\la]$, then
		\begin{equation}\label{soggelp}	\|\1_{(\la-1,\la]}(P)f\|_{L^p(M)}\ls \la^{\sigma(p) }\|f\|_{L^2(M)},\ \ 2\le p\le\infty,\end{equation}
		where
		\begin{equation}\label{sigmap}
			\sigma(p)=\begin{cases}\frac{n-1}2(\frac12-\frac1p),\ \ \  2\le p<\frac{2(n+1)}{n-1}\\
				\frac{n-1}2-\frac np,\ \ \ \ \ \  \frac{2(n+1)}{n-1}\le p\le \infty.\end{cases}
		\end{equation}
		See \cite[Theorem 5.1.1]{SoggeFIO}. The $L^\infty$ bound is due to H\"ormander \cite{hor68}. The estimates are sharp and they can be saturated by some spherical harmonics on the standard sphere. So we have the following  interior decay estimates for Steklov eigenfunctions on smooth manifolds with smooth boundary. 
		\begin{corollary}\label{smoothdecay}
			Let $0 \le \delta_0 < \delta$, $0<c<1$, $N>0$ and $\la>0$. Then for $2\le p \le \infty$ and $t\in [0,\delta_0]$ we have
			\begin{equation}\label{smoothup3}
				\|u_\la\|_{L^p(\Sigma_t)} \ls (\la^{\sigma(p)}e^{-c\la G(t)} + \la^{-N})\|e_\la\|_{L^2(M)},\ \forall N.
			\end{equation}
			This implies
			\begin{equation}\label{smoothup4}
				\|u_\la\|_{L^p(\Sigma_t)} \ls \la^{\sigma(p)}(1+\la t)^{-N}\|e_\la\|_{L^2(M)},\ \forall N.
			\end{equation}
			Here $\sigma(p)$ is Sogge's exponent in \eqref{soggelp}.
		\end{corollary}
		Our estimates \eqref{smoothup3} and \eqref{smoothup4} are at least  sharp for small $t$, by the sharpness of Sogge's estimates \eqref{soggelp}. We also have the derivative estimates for $2\le p \le \infty$
		\begin{equation}\label{smoothup5}
		\|\partial_{t,x}^\alpha u_\la\|_{L^p(\Sigma_t)} \ls \la^{\sigma(p)+|\alpha|}(1+\la t)^{-N}\|e_\la\|_{L^2(M)},\ \forall \alpha,\forall N.
		\end{equation}
		
		The interior decay properties of Steklov eigenfunctions have been studied by Hislop-Lutzer \cite{HL2001}, Polterovich-Sher-Toth \cite{PST19},  Galkowski-Toth \cite{GT2019},  Di Cristo-Rondi \cite{DR2019},  Daudé-Helffer-Nicoleau \cite{DHN21}, Helffer-Kachmar \cite{HK22}, Levitin-Parnovski-Polterovich-Sher \cite{LPPS22}. Now we go over the previous results and compare them with ours.

	In 2001, Hislop-Lutzer  \cite{HL2001} proved when the boundary is smooth, Steklov eigenfunctions decay super-polynomially into the interior, namely \begin{equation}\label{hlup}
			\|u_\la\|_{H^1(K)}=O(\la^{-N}),\ \forall N,
		\end{equation} on any compact subset $K$ of the interior of  $\Omega$, and they also conjectured that the decay is actually of order $e^{-\la t}$ if $\Omega$ and $M$ are real analytic. Under the analyticity assumptions, recently Polterovich-Sher-Toth \cite{PST19} ($\dim\Omega=2$) and Galkowski-Toth \cite{GT2019} ($\dim\Omega\ge2$) established pointwise bounds with exponential decay. Specifically, Galkowski-Toth \cite[Theorem 1]{GT2019} obtained
		\begin{equation}\label{gtup}
			\|u_\la\|_{L^\infty(\Sigma_t)} \ls \la^{\sigma(\infty)+\frac{1}{4}}e^{-\la (t-ct^2)}\|u_\la\|_{L^2(M)}.
		\end{equation}
		The proof  are based on analytic microlocal analysis and thus
		use the analyticity hypotheses in extremely strong ways. So one would perhaps believe that the exponential decay property
		cannot hold in the smooth case.
		For smooth manifolds, Helffer-Kachmar \cite[Theorem 1]{HK22} proved 
		\begin{equation}\label{hkup}
			\|u_\la\|_{L^\infty(\Sigma_t)} \ls \la^{n-1}(1+\la t)^{-N}\|e_\la\|_{L^2(M)},\ \forall N.
		\end{equation}
		Di Cristo-Rondi \cite{DR2019} considered the problems of interior decay for very general elliptic equations. The following estimates can be derived from \cite[Theorem 3.3]{DR2019}
		\begin{equation}\label{drup}
			\|u_\la\|_{L^2(\Sigma_t)} \ls (1+\la t)^{-\frac12} \|e_\la\|_{L^2(M)}.
	\end{equation}
		In comparison, our results \eqref{smoothup3} and \eqref{smoothup4} are better than \eqref{gtup} when $t\ls \la^{-1}\log\la$,  and we do not need the analyticity hypotheses. Moreover, our results are stronger than \eqref{hkup} and \eqref{drup}. By the maximum principle, our results imply \eqref{hlup}.

		%		\begin{equation}
			%			\|u\|_{L^2(\Sigma_t)} \ls e^{c_1 t} h(c_2 t\Phi(u)) \|u\|_{L^2(M)}.
			%		\end{equation}
		%		here $\Phi(u)$ is the lower frequency and defined as 
		%		$$\Phi(u)=\frac{\|u\|_{L^2(M)}}{\|u\|_{H^{-\frac{1}{2}}(M)}}$$
		%		and $h(s)=\min\{e^{-s},(es)^{-1}\}$.

		\subsection{Comparable interior and boundary norms}
		By Theorem \ref{thm2}, if the spectral frequency is bounded below by $\la$, i.e. $f=\1_{\ge \la}(P)f$, then we can integrate in $t$ to get the following interior norm estimates for $1<p\le\infty$
		\begin{equation}\label{intLpup}
			\|\mathcal{H}f\|_{L^p(\Omega)}\ls \la^{-\frac1p}\|f\|_{L^p(M)}.
		\end{equation}
		Next, we further prove that if the boundary datum $f$ has spectral frequency $\approx \la$, then the interior norm of its harmonic extension $\mathcal{H}f$ is comparable to the norm of $f$ on the boundary up to a factor of the frequency. As before, 	let $P\in OPS^1(M)$ be classical and self-adjoint with principal symbol $p(x,\xi)=|\xi|_{g(x)}$. 
		
		\begin{theorem}\label{thm3}
			Let $\la>0$ and $\beta \in C_0^{\infty}(\mathbb{R}^+)$. Let $f_\la=\beta(P/\la)f$. Then for any $1\le p \le \infty$, we have
			\begin{equation}
				\|\mathcal{H}f_\la\|_{L^p(\Omega)} \approx \la^{-\frac1p}\|f_\la\|_{L^p(M)} ,
			\end{equation}
			$i.e.$ there is a constant $C=C(\Omega)>0$ such that
			\begin{equation}
				C^{-1}\la^{-\frac1p}\|f_\la\|_{L^p(M)} \le \|\mathcal{H}f_\la\|_{L^p(\Omega)} \le C\la^{-\frac1p}\|f_\la\|_{L^p(M)}.
			\end{equation}
		\end{theorem}
		The power gain $\la^{-\frac1p}$  is  due to the rapid decay property  in $\la t$, like $(1+\la t)^{-N}$.  This result is sharp in the sense that it can be saturated by Steklov eigenfunctions on Euclidean balls.  As a corollary, we have the following two-sided interior $L^p$ estimates for Steklov eigenfunctions.
	
	\begin{corollary}
		For $1\le p\le \infty$, we have 
		\begin{equation}\label{eqnorm}
			\|u_\la\|_{L^p(\Omega)}\approx \la^{-\frac1p}\|e_\la\|_{L^p(M)}.
		\end{equation}
	\end{corollary}
Motivated by the study of nodal sets, the upper bounds in \eqref{eqnorm} were first obtained in our joint work with Huang and Sire \cite{HSWZ2023} for generalized Steklov eigenfunctions with rough potentials.

\subsection{Almost-orthogonality, transversal restrictions and bilinear estimates}We denote the interior and boundary inner products as
\[\langle u,v \rangle_\Omega=\int_\Omega u\bar{v}\ \ \ \ \text{and} \ \ \ \ \langle f,g \rangle_M=\int_M f\bar{g}.\] It is easy to see the gradients of Steklov eigenfunctions are orthogonal on $\Omega$. Indeed,
$$\langle \nabla u_{\la},\nabla u_{\mu} \rangle_\Omega = \langle \mathcal{N} e_{\la},e_{\mu} \rangle_M = 0, \ \forall \la\neq \mu.$$
But for general $\Omega$ other than Euclidean balls,  Steklov eigenfunctions are  not orthogonal  in $L^2(\Omega)$. However, we can show they are almost-orthogonal. 	Let $\chi\in \mathcal{S}(\mathbb{R})$ satisfy $$\supp \hat \chi\subset (\frac12,1),\ \ \chi(0)=1.$$ As before, 	let $P\in OPS^1(M)$ be classical and self-adjoint with principal symbol $p(x,\xi)=|\xi|_{g(x)}$. 
\begin{theorem}[Almost-orthogonality]\label{thm4}
	Let $v_\la = \mathcal{H}\circ\chi(P-\la)f,\ w_\mu=\mathcal{H}\circ\chi(P-\mu)g$. We have
	\begin{equation}\label{inp}
		|\langle v_{\la}, w_{\mu} \rangle_\Omega| \ls (1+|\la+\mu|)^{-1}(1+|\la-\mu|)^{-N}\|f\|_{L^2(M)}\|g\|_{L^2(M)},\ \forall N.
	\end{equation}
\end{theorem}
The crux of the proof is the kernel formula in Lemma \ref{intkernel}. The question on the $L^2$-orthogonality of Steklov eigenfunctions has been considered by Auchmuty \cite{A04}, Auchmuty-Rivas \cite{AR16}, Cho-Rivas \cite{CR22}. Specifically, Cho-Rivas \cite{CR22} analyzed the problem on rectangles by the explicit formula for the inner products in $L^2(\Omega)$, and provide accompanying numerics. However, a complete description of the orthogonality in $L^2(\Omega)$ has not been made so far. We address this issue in the following corollary. 
\begin{corollary}For Steklov eigenfunctions $u_\la$ and $u_\mu$, we have
\begin{equation}
	|\langle u_{\la}, u_{\mu} \rangle_\Omega| \ls (1+\la+\mu)^{-1}(1+|\la-\mu|)^{-N}\|e_\la\|_{L^2(M)}\|e_\mu\|_{L^2(M)},\ \ \forall N.
\end{equation}
\end{corollary}
Thus, by Theorem \ref{thm3} we have for $\mu\ge\la\ge1$
\begin{equation}
	|\langle u_{\la}, u_{\mu} \rangle_\Omega| \ls (\la/\mu)^{\frac12}(1+|\la-\mu|)^{-N}\|u_\la\|_{L^2(\Omega)}\|u_\mu\|_{L^2(\Omega)},\ \ \forall N.
\end{equation}
This significantly improves the trivial bounds by Cauchy-Schwarz, and provide a rigorous justification for the almost-orthogonality of Steklov eigenfunctions in $L^2(\Omega)$.

Similarly, we can establish the  bilinear estimates and transversal restriction estimates.
\begin{theorem}[Bilinear estimates]\label{thm5}
	Let $v_\la = \mathcal{H}\circ\chi(P-\la)f,\ w_\mu=\mathcal{H}\circ\chi(P-\mu)g$ with $\mu\ge\la \ge 1$. Then we have
	\begin{equation}\label{bilin}
		\| v_{\la} w_{\mu} \|_{L^2(\Omega)} \ls \mu^{-\frac12}B\|f\|_{L^2(M)}\|g\|_{L^2(M)},
	\end{equation}
	where
	\begin{equation}
		B=\begin{cases}
			\la^\frac14,\ \ \ \ \ \ \ \ \ \ n=2\\
			\la^{\frac12}\sqrt{\log\la},\ \ n=3\\
			\la^{\frac{n-2}2},\ \ \ \ \ \ \  n\ge4.\end{cases}
	\end{equation}
\end{theorem}
Bilinear and multilinear eigenfunction estimates on closed manifolds have been proved by Burq-G\'erard-Tzvetkov \cite{BGTbi,BGTbi2}. Our result \eqref{bilin} shows that the same type of estimates still hold for harmonic functions. In comparison to the results in \cite{BGTbi,BGTbi2}, there is a power gain $\mu^{-\frac12}$ that only depends on the relatively high frequency $\mu$. This is due to the rapid decay property of harmonic functions. Moreover, the bound \eqref{bilin} is sharp in the sense that it can be saturated by Steklov eigenfunctions on Euclidean balls. 
\begin{corollary}
	If $B$ is defined as in Theorem \ref{thm6}, then we have the bilinear estimates for Steklov eigenfunctions
	\begin{equation}
		\|u_\la u_\mu\|_{L^2(\Omega)}\ls \mu^{-\frac12}B\|e_\la\|_{L^2(M)}\|e_\mu\|_{L^2(M)}.
	\end{equation}
\end{corollary}
Equivalently, by Theorem \ref{thm3} we have for $\mu\ge\la\ge1$
\begin{equation}
	\| u_{\la} u_{\mu} \|_{L^2(\Omega)}\ls \la^{\frac12}B\|u_\la\|_{L^2(\Omega)}\|u_\mu\|_{L^2(\Omega)}.
\end{equation}
The bound $\la^\frac12 B$ is sharp on on Euclidean balls and only depends on the relatively low frequency $\la$.

\begin{theorem}[Transversal restriction estimates]\label{thm6}
	Let  $0\le k\le n-1$. Let $\Sigma$ be a $(k+1)$-dimensional submanifold in $\Omega$, intersecting the boundary $\partial\Omega=M$ transversally, with $\partial \Sigma\subset \partial\Omega$. For $\la\ge1$,  let $v_\la = \mathcal{H}\circ\chi(P-\la)f$. Then we have
	\begin{equation}\label{Qbgt}	\|v_\la\|_{L^p(\Sigma)}\ls \la^{-\frac1p}A\|f\|_{L^2(M)},\ \ 2\le p\le\infty.\end{equation}
	When $n\ge2$,
	\begin{itemize}
		\item If $k\le n-3$, then $A= \la^{\frac{n-1}2-\frac{k}p}$ for $2\le p\le \infty$
		\item If $k=n-2$, then $A= \la^{\frac{n-1}2-\frac{k}p}$ for $2< p\le \infty$, and $A= \la^{\frac12}(\log\la)^\frac12$ for $p=2$
		\item If $k=n-1$, then $A= \la^{\frac{n-1}2-\frac{k}p}$ for $\frac{2n}{n-1}\le p\le \infty$, and $A= \la^{\frac{n-1}4-\frac{n-2}{2p}}$ for $2\le p<\frac{2n}{n-1}$
	\end{itemize}
	and when $n=1$ and $k=0$, we have $A=1$ for $2\le p\le \infty$.
\end{theorem}
Eigenfunction restriction estimates on closed manifolds were obtained by Burq-G\'erard-Tzvetkov \cite{BGT} and Hu \cite{hu}. See also Tataru \cite{tataru} for the first appearance of
this type of estimates. The transversal restriction estimates of harmonic functions on hypersurfaces can be found  in Taylor \cite[Exercise 4 in Section 11 of Chapter 7 ]{TaylorPDE2}. Our result \eqref{Qbgt} precisely determines the dependence on the spectral frequency, and it is sharp in the sense that it can be saturated by Steklov eigenfunctions on Euclidean balls.

\begin{corollary}
	If $\Sigma$ and $A$ are defined as in Theorem \ref{thm5}, then we have the transversal restriction estimates for Steklov eigenfunctions
	\begin{equation}
		\|u_\la\|_{L^p(\Sigma)}\ls \la^{-\frac1p}A\|e_\la\|_{L^2(M)}.
	\end{equation}
\end{corollary}
Equivalently, by Theorem \ref{thm3} we have 
\begin{equation}
	\| u_{\la}\|_{L^p(\Sigma)}\ls \la^{\frac12-\frac1p}A\|u_\la\|_{L^2(\Omega)}.
\end{equation}
The bound $\la^{\frac12-\frac1p}A$ is sharp on Euclidean balls.

\subsection{Organization of the paper} In Section 2, we introduce the symbol classes and the Dirichlet-to-Neumann operator. In Section 3, we prove the Theorem \ref{thm1} and its sharpness. In Section 4, we construct a parametrix for the Poisson integral operator and calculate the kernel of the its composition with the spectral cluster. In Section 5, we prove Theorem \ref{thm2} and Theorem \ref{thm3} by the parametrix in Lemma \ref{parametrix}. In Section 6, we prove Theorem \ref{thm4}  by Lemma \ref{intkernel} and give a proof sketch for Theorems \ref{thm5} and \ref{thm6}, and discuss related results on the numerical approximation of harmonic functions with various boundary conditions. In the Appendix, we prove Lemma \ref{pdolemma}.
\subsection{Notations}Throughout this paper, $X\ls Y$ means $X\le CY$  for some positive constants $C$ that are independent of key parameters: the distance $t$ to the boundary, and the frequencies $\la$, $\mu$. If $X\ls Y$ and $Y\ls X$, we denote $X\approx Y$.

%		Our second result concerns the upper bound $t$-Restriction $L^p$ norm for $2 \le p < \infty$ on analytic manifolds.
%		
%		\begin{theorem}\label{upbound}
%			Let $(\Omega^{n+1},h)$ be analytic manifold with analytic boundary $(M,g)$, for any $\eps > 0$, there are constants $\delta_0(\eps)>0$, $C(\eps)>0$ and $\Lambda(\eps) > 0$ depends on $(\Omega,h)$ and $\eps$, such that for any $0 \le t \le \delta_0$ and $\la > \Lambda$, the following $t$-Restriction upper bound holds
%			\begin{equation}\label{upstimate}
%				\|u_\la\|_{L^p(\Sigma_t)} \le C\la^{\sigma(n,p)}e^{-(1-\eps)t\la}\|u_\la\|_{L^2(M)}.
%			\end{equation}
%			where
%			\[
%			\sigma(n,p)=\begin{cases}\frac{n-1}2(\frac12-\frac1p),\ \ \  2\le p<\frac{2(n+1)}{n-1}\\
%				\frac{n-1}2-\frac np,\ \ \ \ \ \  \frac{2(n+1)}{n-1}\le p\le \infty.\end{cases}\]	
%		\end{theorem}
%		
%		Modulo the $\eps$ loss, this  estimate is saturated by balls in $\Rn$, see Example \ref{exBall} in next section.

		\section{Preliminaries}
		
		In this section, we introduce some background on pseudo-differential operators. Let $D_x=\frac1 i\partial_x$.
		
	Symbol class $S^m$ is the set of smooth functions $p(x,\xi)$ satisfying
		\begin{equation}\label{regsymbol}
			|D_x^{\beta} D_{\xi}^{\alpha} p(x,\xi) |\le C_{\alpha\beta}(1+|\xi|)^{m-|\alpha|},\; \forall \alpha,\beta.
		\end{equation}
		Let $OPS^m$ be the set of pseudo-differential operators with symbols in $S^m$. Let $\delta_0>0$. We call $p(t,x,\xi) \in C^{\infty}([0,\delta_0]\times\Rn\times\Rn)$ a smooth family of symbols in $S^m$ if 
		\begin{equation}\label{smbounds}
			|D^j_tD_x^{\beta} D_{\xi}^{\alpha} p(t,x,\xi) |\le C_{j\alpha\beta}(1+|\xi|)^{m-|\alpha|},\; \ t\in [0,\delta_0],\ \forall j,\alpha,\beta.
		\end{equation}
		The associated family of operators are called a smooth family in $OPS^m$.
		
		Symbol class $\mathcal{P}^{m}$ is the set of functions $p(t,x,\xi)$ in $C^{\infty}([0,\delta_0]\times\Rn\times\Rn)$ such that 
		\[|D_x^{\beta} D_{\xi}^{\alpha}(t^k D_t^\ell p(t,x,\xi))| \le C_{\alpha\beta}(1+|\xi|)^{m-k+\ell-|\alpha|},\; t\in [0,\delta_0],\ \forall k,\ell,\alpha,\beta.\]
		Symbol class $\mathcal{P}^{m}_e$ is the set of functions $p(t,x,\xi) \in C^{\infty}([0,\delta_0]\times\Rn\times\Rn)$ such that
		\begin{equation}\label{psymbol2}
			p(t,x,\xi) = q(t,x,\xi)e^{-ct\langle \xi\rangle}
		\end{equation}
		with $\langle \xi\rangle=\sqrt{1+|\xi|^2}$ for some $c > 0$ and $q(t,x,\xi) \in \mathcal{P}^{m}$. See e.g. Taylor \cite[Chapter 7]{TaylorPDE2}.

%		
%		\begin{definition}
%			Manifolds with multiples ends.
%		\end{definition}
%		When $M=\partial\Omega$ contains multiple connected components $M^1,M^2,...,M^m$, $m \ge 2$. Suppose we have $m$ functions $\{f_j\}_{j=1}^m$, and for each $j$, $f_j$ is defined on component $M^j$, we define $[f_1,f_2,...,f_m]_M$ as the function $f$ on $M$ such that its restriction to $M^j$ equals $f_j$, for $1 \le j \le m$. Similarly, we can define the $t$ sets with respect to each single component, $\Omega^j_t = \{x\in \Omega: dist(x,\partial \Omega)\ge t\}$, $\Sigma^j_t = \{x\in \Omega: dist(x,\partial \Omega)=t\}$. We want to mention that $\partial \Omega^j_t = \Sigma^j_t \cup (\cup_{i\neq j}M^i)$, and the injective range for map $\phi$ on $M^j$ may be larger than $\delta(\Omega)$.

	%	\begin{definition}
		%	Dirichlet to Neumann operator.
	%	\end{definition}
		For any $f\in L^2(M)$, we denote its harmonic extension on $\Omega$ by $u=\mathcal{H}f$ that solves the boundary value problem \eqref{lapdirichlet}.
		Then the DtN operator on $M$ is defined as $\mathcal{N}f = \partial_\nu u|_M$, where $\nu$ is the outward unit norm vector. In addition, $\mathcal{N}$ is self-adjoint.

Using the Fermi coordinates, for each $0 \le t \le \delta_0<\delta$, we can define a DtN operator $\mathcal{N}_t$ on $\Sigma_t$. By \cite[Section C of Chapter 12]{TaylorPDE2}, the Laplacian $\Delta$ on $\Omega$ in local coordinates with respect to the boundary has the form
\begin{equation}\label{lapdec0}\Delta=(\partial_t-A_1(t))(\partial_t+A(t)) +B(t)\end{equation}
where the smooth families $A(t),A_1(t)\in OPS^1(M)$ and $B(t)\in OPS^{-\infty}(M)$. Moreover, $A(t)=\mathcal{N}_t$ modulo smoothing operators, then $\mathcal{N}_t$ is a smooth family in $OPS^1(M)$. By \cite[Proposition C.1 in Chapter 12]{TaylorPDE2}, we have explicit formulas for the principal and sub-principal symbols of $\mathcal{N}_t$.

%By extending $\Omega$ to a suitable closed manifold without boundary such that $\Omega^c$ is connected, for example, we can glue  copy of $\Omega$ along with the boundary and extend the metric smoothly to this copy, note we don't require the new metric to be symmetric with respect to the boundary. By the method of layer potentials, one has $S(t)$ and $N(t)$ as two smooth families in $OPS^{-1}$, such that $S(t): H^{s}\rightarrow H^{s-1}$ are isomorphisms and
%		\[ \mathcal{N}_t S(t) = \frac{1}{2}(I-N(t)),\]
	
		\begin{lemma} \label{symblem}
			The Dirichlet-to-Neumann operator satisfies
			\[\mathcal{N}_t=  \sqrt{-\Delta_{g_t}}-\mathcal{B}_t\]
			where $\mathcal{B}_t\in OPS^{0}(\Sigma_t)$ has principal symbol
			\begin{equation}
				\sigma_t(x,\xi)=\frac12(\text{Tr}\mathcal{W}_{t,x}- \frac{\langle \mathcal{W}_{t,x}^*\xi,\xi\rangle_{t} }{\langle\xi,\xi\rangle_{t} }).
			\end{equation}
		\end{lemma}
	 Here we define DtN operator by using the outward unit normal vector, so there is a difference in the sign compared to the one in \cite[Section C of Chapter 12]{TaylorPDE2}.

%Furthermore, by Sogge's $L^p$ estimate for spherical harmonics, see \cite{Sogge1985}, there is a sequence of $\{e_j\}$ such that
%\[  \|e_j\|_{L^p(M)}  \approx \la_j^{\sigma(n,p)}\|e_j\|_{L^2(M)},\]
%thus 
%\begin{equation}
%	\|u_j\|_{L^p(\Sigma_t)} \approx \la_j^{\sigma(n,p)}e^{-\la_j K(t)}\|e_j\|_{L^2(M)} = \la_j^{\sigma(n,p)}e^{-\la_j (t + R^{-1}t^2+O(t^3))}\|e_j\|_{L^2(M)}. 
%\end{equation}

\section{Exponential lower bounds}

We first prove Theorem \ref{thm1} based on the frequency function in  Almgren \cite{A79} and Garofalo-Lin \cite{GL1986}, and microlocal analysis techniques involving Lemma \ref{symblem}. See also Lin \cite{Lin1991} and Di Cristo-Rondi \cite{DR2019}. And then we show the sharpness of Theorem \ref{thm1} by constructing examples.

\subsection{Proof of Theorem \ref{thm1}}
	Let
\begin{equation}\label{hdn}
	H(t) = \int_{\Sigma_t} u^2,\ \  D(t) = -\int_{\Sigma_t} u \partial_t u,\ \  N(t)=D(t)/H(t).
\end{equation}
%When $t=0$, we have $M=\Sigma_0$ and
%	\[  D(0) = -\int_{\Sigma_0} u \partial_t u =  \int_{M} u \partial_\nu u = N(0) \int_{\Sigma_0} u^2\]
%	where $\partial_\nu$ is the outward normal derivative at the boundary, and $\partial_t$ is inward. 
	Let $(t,x)$  be the Fermi coordinates described before and suppose $0 \le t \le \delta_0 <\delta$.  We can identify each $\Sigma_t$ as $M$ with metric tensor $g_t$, so we can rewrite $H(t),\ D(t)$ as
	\[  H(t) = \int_{\Sigma_t} u^2=\int_{M} u(t,x)^2 d\omega_t(x)  \]
	\[  D(t) = -\int_{\Sigma_t} u \partial_t u = -\int_{M} u(t,x) \partial_t u(t,x) d\omega_t(x) \]
	where $d\omega_t(x) = \sqrt{\det(g_{ij}(t,x))}dx$ in local coordinates. Then
	\begin{equation}
		\begin{aligned}
			H'(t) &= 2\int_{M} u(t,x)\partial_t u(t,x) d\omega_t(x) +  \int_{M} u^2(t,x) \partial_t(d\omega_t(x)).
		\end{aligned}
	\end{equation}
	
	For the volume derivative term $\partial_t(d\omega_t(x))$, As in \cite[(C.26) in Section C of Chapter 12]{TaylorPDE2},  under Fermi coordinate, we have
	\begin{equation}\label{derdet}
		\frac{\partial_t\det(g_{ij}(t,x))}{\det(g_{ij}(t,x))} = -2\text{Tr}\mathcal{W}_{t,x}.
	\end{equation}
	Then we have
	\begin{equation}\label{measureder}
		\begin{aligned}
			\partial_t(d\omega_t(x)) &=\partial_t(\sqrt{\det(g_{ij}(t,x))}dx)\\
			& = \frac{\partial_t(\det(g_{ij}(t,x)))}{2\sqrt{\det(g_{ij}(t,x))}}dx\\
			& = \frac{\partial_t(\det(g_{ij}(t,x)))} {2 \det(g_{ij}(t,x))} d\omega_t(x)\\
			& =-(\text{Tr}\mathcal{W}_{t,x}) d\omega_t(x).
		\end{aligned}
	\end{equation}
	Thus,
	\begin{equation}\label{hder}
		H'(t) = 2\int_{\Sigma_t} u\partial_t u -  \int_{\Sigma_t} (\text{Tr}\mathcal{W}_{t,x}) u^2=-2D(t) -  \int_{\Sigma_t} (\text{Tr}\mathcal{W}_{t,x}) u^2.
	\end{equation}
	Next, we have
	\begin{equation}\label{Dder1}
		D'(t) = -\int_{M} |\partial_t u|^2 d\omega_t(x) - \int_{M} u \partial_t^2 u d\omega_t(x) - \int_{M} u \partial_t u \partial_t(d\omega_t(x)).
	\end{equation}
	By \eqref{lap}, 
	\[0 = \Delta u = \partial_t^2 u(t,x) + a(t,x)\partial_t u(t,x) +\Delta_{g_t} u(t,x)\]
	and we use \eqref{aformula}, \eqref{measureder} to obtain
	\begin{equation}
		\begin{aligned}
			\int_{M} u \partial_t^2 u d\omega_t(x) &= -\int_{M}  u \partial_t u  a(t,x)d\omega_t(x) - \int_{M} u \Delta_{g_t} u  d\omega_t(x)\\
			&= -\int_{M}  u \partial_t u   \frac{\partial_t(\det(g_{ij}(t,x)))} {2 \det(g_{ij}(t,x))}d\omega_t(x) - \int_{M} u \Delta_{g_t} u  d\omega_t(x)\\
			&= -\int_{M} u \partial_t u  \partial_t(d\omega_t(x)) - \int_{M} u \Delta_{g_t} u  d\omega_t(x).
		\end{aligned}
	\end{equation}
	Plugging this into \eqref{Dder1}, we get
	\begin{equation}
		D'(t) = -\int_{\Sigma_t} |\partial_t u|^2 + \int_{\Sigma_t} u \Delta_{g_t} u.
	\end{equation}
	To handle the second term on the right, we use the fact that $\sqrt{-\Delta_{g_t}}$ is self-adjoint and  Lemma \ref{symblem} to obtain
	\begin{equation}
		\begin{aligned}
			\int_{\Sigma_t} u \Delta_{g_t} u &= -\int_{\Sigma_t} |\sqrt{-\Delta_{g_t}} u|^2\\
			&= -\int_{\Sigma_t} | (\mathcal{N}_t+\mathcal{B}_t) u|^2\\
			&= -\int_{\Sigma_t} | \mathcal{N}_t u|^2 -2\int_{\Sigma_t} \mathcal{N}_t u \mathcal{B}_t u - \int_{\Sigma_t} | \mathcal{B}_t u|^2.
		\end{aligned}
	\end{equation}
Since $u$ is harmonic on $\Sigma_t$ and $ \mathcal{N}_t u = \partial_t u$ on $\Sigma_t$, and $\mathcal{N}_t$ is self-adjoint, we have
	\begin{equation}\label{Dder}
		\begin{aligned}
			D'(t) &= -2\int_{\Sigma_t} |\partial_t u|^2 -2\int_{\Sigma_t} \mathcal{N}_t u \mathcal{B}_t u - \int_{\Sigma_t} | \mathcal{B}_t u|^2\\
			&= -2\int_{\Sigma_t} |\partial_t u|^2 -2\int_{\Sigma_t} u \mathcal{N}_t \mathcal{B}_t u - \int_{\Sigma_t} | \mathcal{B}_t u|^2.
		\end{aligned}
	\end{equation}
	Now we are ready to compute the derivative of $N(t)$. Since $\mathcal{B}_t \in OPS^0$, it is bounded on $L^2$ and then $\int_{\Sigma_t} | \mathcal{B}_t u|^2/\int_{\Sigma_t} |u|^2=O(1)$. So by \eqref{hdn}, \eqref{hder} and \eqref{Dder} we have
	\begin{align}\nonumber
		N'(t) &= \frac{D'(t)}{H(t)} - \frac{D(t)H'(t)}{H(t)^2}\\ \nonumber
		&= \frac{-2\int_{\Sigma_t} |\partial_t u|^2-2\int_{\Sigma_t} u \mathcal{N}_t \mathcal{B}_t u}{\int_{\Sigma_t} u^2} + \frac{2(\int_{\Sigma_t} u\partial_t u)^2}{(\int_{\Sigma_t} u^2)^2} + \frac{\int_{\Sigma_t} (\text{Tr}\mathcal{W}_{t,x})u^2}{\int_{\Sigma_t} u^2}N(t)+O(1)\\ \label{Nprime}
		&= \frac{-2\int_{\Sigma_t} |\partial_t u|^2 \int_{\Sigma_t} u^2 + 2(\int_{\Sigma_t} u\partial_t u)^2}{(\int_{\Sigma_t} u^2)^2} - \frac{2\int_{\Sigma_t} u \mathcal{N}_t \mathcal{B}_t u}{\int_{\Sigma_t} u^2} + \frac{\int_{\Sigma_t} (\text{Tr}\mathcal{W}_{t,x}) u^2}{\int_{\Sigma_t} u^2}N(t)+O(1)\\ \nonumber
		&\le  - \frac{2\int_{\Sigma_t} u \mathcal{N}_t \mathcal{B}_t u}{\int_{\Sigma_t} u^2} + \frac{\int_{\Sigma_t} (\text{Tr}\mathcal{W}_{t,x}) u^2}{\int_{\Sigma_t} u^2}N(t)+O(1). \nonumber
	\end{align}
	The last step follows from Cauchy–Schwartz inequality.
	
	For the rest terms, by Lemma \ref{symblem}, the principal symbol  of $-2\mathcal{N}_t \mathcal{B}_t \in OPS^{1}$ is
	\[-2|\xi|_{g_t}\sigma_t(x,\xi) \le |\xi|_{g_t}\tilde{k}_t(x) \le |\xi|_{g_t} \sup_{x\in\Sigma_t}\tilde{k}_t(x).\]
	So the symbol of $(\sup_{x\in\Sigma_t}\tilde{k}_t(x))\mathcal{N}_t + 2\mathcal{N}_t \mathcal{B}_t$ is the sum of  a non-negative principal symbol and a symbol of order 0. Combining with the $L^2$ boundedness of operators of order 0, we can apply the sharp G$\mathring{\text{a}}$rding inequality \cite[Theorem 18.1.14]{hor3} to the principal symbol and obtain
	\[- \frac{2\int_{\Sigma_t} u \mathcal{N}_t \mathcal{B}_t u}{\int_{\Sigma_t} u^2} \le N(t)\sup_{x\in\Sigma_t}\tilde{k}_t(x) + O(1).\]
	It is easy to see
	\begin{equation}
		\frac{\int_{\Sigma_t} (\text{Tr}\mathcal{W}_{t,x}) u^2}{\int_{\Sigma_t} u^2}N(t) \le N(t)\sup_{x\in\Sigma_t}\text{Tr}\mathcal{W}_{t,x}
.	\end{equation}
	We finally have
	\begin{equation}\label{frder}
		N'(t) \le (\sup_{x\in\Sigma_t}\tilde{k}_t(x) + \sup_{x\in\Sigma_t}\text{Tr}\mathcal{W}_{t,x})N(t) + O(1) = \Theta(t)N(t) +O(1).
	\end{equation}
	So we have,
	\begin{equation}
		N(t) \le e^{\int_{0}^{t}\Theta(y) dy}N(0) + O(1).
	\end{equation}
	By \eqref{hder}, we have
	\[(\log H(t))' = \frac{H'(t)}{H(t)} \ge -2N(t) - \sup_{x\in\Sigma_t}\text{Tr}\mathcal{W}_{t,x}\]
	again integrating both sides
	\begin{equation}
		\begin{aligned}
			\log H(t) - \log H(0) &\ge -2\int_{0}^{t}N(s)ds - \int_{0}^{t} \sup_{x\in\Sigma_s}\text{Tr}\mathcal{W}_{s,x}ds\\
			&\ge -2N(0)\int_{0}^{t}e^{\int_{0}^{t}\Theta(y) dy}dt +O(t)\\
			%	&= -2\la K(t) - C_5(\delta_0)
		\end{aligned}
	\end{equation}
	and this gives us the desired bound
	\[  \|u\|_{L^2(\Sigma_t)} \gs e^{-N(0) K(t) }\|u\|_{L^2(M)} .\]

\subsection{Sharpness of Theorem \ref{thm1}}
Now we are going to show the exponential lower bound can be saturated by Euclidean balls and a family of symmetric warped product manifolds.

Let $\Omega = [-R,R] \times M_0$, where $(M_0,g)$ is a connected closed manifold, and $\Omega$ is equipped with the metric $h=\rho(s)^2 g + ds^2$ for some positive $\rho \in C^{\infty}([-R,R])$. First, we shall compute some geometric quantities and the Weingarten map for some special cases. 

Next, we compute Weingarten maps of manifolds $(\Omega,h)$ in the setting as in Proposition \ref{symbound} without the symmetry condition $\rho(s) = \rho(-s)$. For $0 \le t < R$,  $$\Sigma_t=\{-R+t\}\times\M_0\cup \{R-t\}\times M_0 = \Sigma_t^-\cup \Sigma_t^+.$$ We shall only compute it at $\Sigma_t^+$, under this setting $g_t=\rho(R-t)^2g$. 

As in  \cite[(4.68)–(4.70) of Appendix C]{TaylorPDE2}, we consider the one-parameter family of submanifolds given by $$\phi_\tau:\Sigma_t^+\rightarrow\Omega,\  \phi_\tau(R-t,x)=(R-t-\tau,x),\ \forall x\in M_0.$$ Let $\tilde g_\tau$ be the induced metric on the submanifold $\phi_\tau(\Sigma_t^+)$ and $x_+=(R-t,x)$. Then for any tangent vectors $v_1,\ v_2$ on $\Sigma_t^+$ at $x_+$, we have
\[\partial_\tau \tilde g_\tau(v_1,v_2)|_{\tau=0} = -2\langle \partial_t,\mathrm{I\!I}_{x_+}(v_1,v_2)\rangle=-2\langle\mathcal{W}_{t,x_+}v_1,v_2\rangle_t\]
where the second fundamental form $\mathrm{I\!I}_{x_+}(v_1,v_2) = \langle\mathcal{W}_{t,x_+}v_1,v_2\rangle_t\partial _t$, and $\partial_t$ is the inward unit normal vector on $\Sigma_t^+$.

Since $\tilde g_\tau = \rho(R-t-\tau)^2g$, we have
\[
\begin{aligned}
	\partial_\tau \tilde g_\tau(v_1,v_2)|_{\tau=0} 	&= -2\rho'(R-t)\rho(R-t)g(v_1,v_2)\\
	&=-\frac{2\rho'(R-t)}{\rho(R-t)}g_t(v_1,v_2)\\
	&=-\frac{2\rho'(R-t)}{\rho(R-t)}\langle v_1,v_2\rangle_t.
\end{aligned}
\]
Thus, for $x\in \Sigma_t^+$
\begin{equation}
	\mathcal{W}_{t,x} = \frac{\rho'(R-t)}{\rho(R-t)}I
\end{equation}
where $I$ is the identity matrix. Similarly for $x \in \Sigma_t^-$,
\[\mathcal{W}_{t,x} = -\frac{\rho'(-R+t)}{\rho(-R+t)}I.\]
So the principal curvatures $k_1(x),\ \cdots,\ k_n(x)$ are all equal to $ \frac{\rho'(R-t)}{\rho(R-t)}$  on $\Sigma_t^+$, and $-\frac{\rho'(-R+t)}{\rho(-R+t)}$ on $\Sigma_t^-$, and
\[\Theta(t) = \max\Big\{\frac{\rho'(R-t)}{\rho(R-t)},-\frac{\rho'(-R+t)}{\rho(-R+t)}\Big\}.\]
Specially if $\rho$ is symmetric about $0$, then
\begin{equation}\label{symqty1}
	k_1(x)=k_2(x)=...=k_n(x)=\frac{\rho'(R-t)}{\rho(R-t)}, \ \ \ \ \tilde{k}_t(x) = -\frac{(n-1)\rho'(R-t)}{\rho(R-t)},
\end{equation}
and
\begin{equation}\label{symqty2}
	\sigma_t(x,\xi) = \frac{(n-1)\rho'(R-t)}{2\rho(R-t)},\;\ \ \ \text{Tr}\mathcal{W}_{t,x} = \frac{n\rho'(R-t)}{\rho(R-t)},\; \ \ \ \Theta(t)  = \frac{\rho'(R-t)}{\rho(R-t)}
\end{equation}
for all $(x,\xi)\in T^*\Sigma_t$. Thus,
\begin{equation}
	K(t)=\int_0^t \frac{\rho(R)}{\rho(R-s)}ds.
\end{equation}

\begin{example}\label{exBall}
	Euclidean balls.
\end{example}
Let $\Omega = B(0,R)\in \R^{n+1}$ be the ball of radius of $R$ centered at the origin. Let $g,\Delta,\mathcal{N}$ be the induced metric, Laplace operator and DtN operator on $\partial\Omega$ (a sphere of radius $R$).  We have
\[ \mathcal{N} = \sqrt{-\Delta+\frac{(n-1)^2}{4R^2}} - \frac{n-1}{2R},\]
see \cite[Section 4 of Chapter 8]{TaylorPDE2}. Under Fermi coordinates, for $0\le t <\delta=R$ we have
\[  g_t = \frac{(R-t)^2}{R^2}g, \; \Delta_t = \frac{R^2}{(R-t)^2}\Delta,\; \mathcal{N}_t = \frac{R}{R-t}\mathcal{N}.\]
Suppose $\{\mu_k\}_{k=0}^{\infty}$ are the eigenvalues (counting multiplicity) of  $\sqrt{-\Delta}$ and $\{e_k\}_{k=0}^{\infty}$ are the corresponding eigenfunctions. Then
\[ \la_k= \sqrt{\mu_k^2+\frac{(n-1)^2}{4R^2}} - \frac{n-1}{2R},\; \ \ \ u_k = (1-\frac{t}{R})^{R\la_k}e_k,\; k=0,1,2,... \]
are Steklov eigenvalues and eigenfunctions. Thus,
\begin{equation}\label{balldecay}
	\|u_k\|_{L^p(\Sigma_t)} = \frac{(R-t)^{n/p}}{R^{n/p}}(1-\frac{t}{R})^{R\la_k}\|e_k\|_{L^p(M)} \approx e^{-\la_k R\log(R/(R-t))}\|e_k\|_{L^p(M)}. 
\end{equation}
In this case, the metric has the form of $ds^2 + \rho(s)^2g$ under polar coordinate with $\rho(s) = \frac{s}{R}$ when $0 < s \le R$, and the boundary has only one component. Thus
\[  \Theta(t) = \frac{\rho'(R-t)}{\rho(R-t)} = \frac{1}{R-t} \]
and 
\[  K(t) = \int_{0}^{t} e^{\int_{0}^{s}(R-\tau)^{-1}d\tau}ds =   R\log(\frac{R}{R-t}), \]
which is exactly the same as the coefficient of $-\la_k$ in \eqref{balldecay}.

\begin{prop}\label{symbound}
Let $\Omega = [-R,R] \times M_0$, where $(M_0,g)$ is a connected closed manifold, and $\Omega$ is equipped with the metric $h=\rho(s)^2 g + ds^2$ for some positive $\rho \in C^{\infty}([-R,R])$ with $\rho(-s) = \rho(s)$. Then for any $\delta_0 < \delta=R$, there is a constant $C>0$ depends on $M_0$, $\rho$ and $\delta_0$, such that, for each Steklov eigenvalue $\lambda$, there exists a Steklov eigenfunction $u_\lambda$ satisfying
	\begin{equation}\label{symest}
		C^{-1}e^{-\la K(t)}\|u_\la\|_{L^p(M)} \le \|u_\la\|_{L^p(\Sigma_t)} \le Ce^{-\la K(t)}\|u_\la\|_{L^p(M)},
	\end{equation}
	for any $1 \le p \le \infty$.
\end{prop}
The upper bound in \eqref{symest} was also obtained in \cite[Theorem 1.5]{DHN21} in a slightly different form. Here we give a direct proof.

\begin{proof}[Proof of Proposition \ref{symbound}]
	By the symmetry of $\rho$, there is a Steklov eigenfunction of the form $$u_\la = b(s)e_\mu(x)$$ associated with the Steklov eigenvalue $\la$, where $b(s)$ is symmetric or anti-symmetric about $0$ and $e_\mu(x)$ is an $L^2$ normalized eigenfunction on $M_0$.
	
	We just need to handle $p=2$, since it follows from  $u_\la = b(s)e_\mu(x)$ that $$\frac{\|u_\la\|_{L^p(\Sigma_t)}}{\|u_\la\|_{L^p(M)}}=\frac{\|u_\la\|_{L^2(\Sigma_t)}}{\|u_\la\|_{L^2(M)}}.$$
	The argument is similar to the proof of Theorem \ref{thm1}. By \eqref{Nprime}, we have
	\begin{equation}\label{Np2}
		N'(t)=R_1(t)+R_2(t)+R_3(t)+O(1),
	\end{equation}where
	\[R_1 = \frac{-2\int_{\Sigma_t} |\partial_t u_\la|^2 \int_{\Sigma_t} u_\la^2 + 2(\int_{\Sigma_t} u_\la\partial_t u_\la)^2}{(\int_{\Sigma_t} u_\la^2)^2},\]
	\[ R_2 = - \frac{2\int_{\Sigma_t} u_\la \mathcal{N}_t \mathcal{B}_t u_\la}{\int_{\Sigma_t} u_\la^2},\;\ \ \ \ \ 
	R_3 = \frac{\int_{\Sigma_t} (\text{Tr}\mathcal{W}_{t,x}) u_\la^2}{\int_{\Sigma_t} u_\la^2}N(t). \]
	We notice $R_1 = 0$, by the symmetry of $b(s)$, and the fact that $u_\la = b(s)e_\mu(x)$.
	
By \eqref{symqty2}, the principal symbol of $\mathcal{B}_t$ is
	\[\sigma_t(x,\xi) = \frac{(n-1)\rho'(R-t)}{2\rho(R-t)} = -\frac{1}{2}\tilde{k}_t(x),\]
	which is constant for all $x,\xi$ on $T^*\Sigma_t$. So $2\mathcal{N}_t \mathcal{B}_t+\tilde{k}_t\mathcal{N}_t \in OPS^0$, and we can write
	\[R_2 =  -\frac{(n-1)\rho'(R-t)}{\rho(R-t)}N(t) + O(1).\]
	Again by \eqref{symqty2}, we  have $$R_3 = \frac{n\rho'(R-t)}{\rho(R-t)}N(t).$$ Thus, by \eqref{Np2} and \eqref{symqty2} we get
	\[\begin{aligned}
		N'(t)
		&= \frac{\rho'(R-t)}{\rho(R-t)}N(t)+O(1)\\
		&=  \Theta(t)N(t) + O(1).
	\end{aligned}\]
	Then
	\begin{equation}
		N(t) = e^{\int_{0}^{t}\Theta(y) dy}N(0) +O(1).
	\end{equation}
	By \eqref{hder}, we have
	\[(\log H(t))' = \frac{H'(t)}{H(t)} = -2N(t) - O(1),\]
	which gives
	\begin{equation}
		\begin{aligned}
			\log H(t) - \log H(0) &= -2N(0)\int_{0}^{t}e^{\int_{0}^{s}\Theta(y) dy}ds + O(1).
		\end{aligned}
	\end{equation}
This implies the desired bound
	\[  \|u_\la\|_{L^2(\Sigma_t)} ^2\approx e^{-2N(0) K(t) }\|u_\la\|_{L^2(M)} ^2=e^{-2 \la K(t) }\|u_\la\|_{L^2(M)} ^2.\]
\end{proof}

Next, we apply Proposition \ref{symbound} to construct some examples with various decaying speeds, compared to original conjetured decaying speed $e^{-\la t}$ by Hislop-Lutzer \cite{HL2001}. These phenomena have been observed by Galkowski-Toth \cite{GT2019, GT2021}.

%		By Taylor's expansion, we can see if $f$ continuous and $t$ small, we have
%		\[I(f)(t) \approx t + \frac{f(0)}{2}t^2 + o(t^2)\]

\begin{example}\label{exTorus}
	Quasi-cylinder with convex boundary.
\end{example}
Let $M_0=\mathbb{R}^n/(2\pi\mathbb{Z}^n)=\mathbb{T}^n$ be the standard flat torus, and $\Omega= [-1,1]\times M_0$ with the metric $h= ds^2 + (1+s^2)dx^2$. Then
\[ \Theta(t) =\frac{\rho'(R-t)}{\rho(R-t)} = \frac{2(1-t)}{1+(1-t)^2}.\]
By Proposition \ref{symbound}, for each Steklov eigenvalue $\la$, we can find some $ \xi \in \mathbb{Z}^n$ such that
 $u_\la = b(s)e^{i\xi\cdot x}$. Then by \eqref{symest}, we get
\[ \|u_\la\|_{L^p(\Sigma_t)} \approx e^{-\la (t+t^2+O(t^3))}\|u_\la\|_{L^p(M)}, \]
which decays faster than $e^{-\la t}\|u_\la\|_{L^p(M)}$.

\begin{example}
	Quasi-cylinder with concave boundary.
\end{example}
Let $M_0=\mathbb{S}^n$ be the standard sphere and $\Omega= [-\frac{\pi}3,\frac{\pi}3]\times M_0$ with the metric $h= ds^2 + (\cos s)dx^2$. Then
\[ \Theta(t) =-\tan(\frac{\pi}{3}-s).\]
By Proposition \ref{symbound} for each Steklov eigenvalue $\la$, we can find a Steklov eigenfunction $u_\la = b(s)e_\mu(x)$ such that
\[ \|u_\la\|_{L^p(\Sigma_t)} \approx  e^{-\la (t-\frac{\sqrt{3}}{2}t^2+O(t^3))}\|u_\la\|_{L^p(M)}, \]
which decays slower than $e^{-\la t}\|u_\la\|_{L^p(M)}$.

		\section{Parametrix for the Poisson integral operator}
	As in Taylor \cite[Section 12 of Chapter 7]{TaylorPDE2}, we can construct a parametrix for the Poisson integral operator $u=\mathcal{H}f$ near the boundary for the boundary value problem
			\begin{equation}\label{hmbdvalue}
			\begin{cases}
				\Delta u(t,x)=0,\ \ (t,x)\in \Omega,\\
				u(0,x)=f(x),\ \ \ x\in \partial\Omega.
			\end{cases}
		\end{equation}
		See also Boutet de Monvel \cite{M71}, Lee-Uhlmann \cite{LU89}.
		
		Let $\mathcal{N}$ be the DtN operator on $M=\partial\Omega$. As in Section 2,  the Laplacian $\Delta$ on $\Omega$ in local coordinates with respect to the boundary has the form
		\begin{equation}\label{lapdec}\Delta=(\partial_t-A_1(t))(\partial_t+A(t)) +B(t)\end{equation}
		where the smooth families $A(t), A_1(t)\in OPS^1(M)$ and $B(t)\in OPS^{-\infty}(M)$. Moreover, $A(0)=\mathcal{N}$ modulo a smoothing operator, and $A(t)$ has the symbol $a(t,x,\xi)\in S^1$ with principal symbol $|\xi|_{g_t(x)}$. We have
		\[|\partial_t^k\partial_x^\beta\partial_\xi^\alpha a(t,x,\xi)|\le C_{\alpha,\beta,k} (1+|\xi|)^{1-|\alpha|},\ \ \forall \alpha,\beta,k.\]
		
		\begin{lemma}[Parametrix for the Poisson integral operator]\label{parametrix}
		
			For any $0<\delta_0 < \delta$ and $t\in [0,\delta_0]$,  under Fermi coordinates, there exist a parametrix $P_t$ locally given by
			\begin{equation}\label{paraformula}
				P_tf(x) = \frac{1}{(2\pi)^n}\int e^{ix\cdot \xi}b(t,x,\xi)\hat{f}(\xi)d\xi, \ b \in \mathcal{P}^0_e,			\end{equation}
			and a smooth family of smoothing operators $R_t$ such that the Poisson integral operator
			\begin{equation}
				\mathcal{H}f(t,x) = P_tf(x)+R_tf(x).
			\end{equation}
			Moreover, the symbol
			\[b(t,x,\xi)\sim b_0(t,x,\xi)+\sum_{j\ge1}\sum_{2\le k\le 2j}b_0(t,x,\xi)t^k q_{j,k}(t,x,\xi)\]
			with $q_{j,k}(t)$ are smooth families in $S^{k-j}$ and $b_0(t,x,\xi)=e^{-\int_0^t a(s,x,\xi)ds}\in \mathcal{P}_e^0$. 
		\end{lemma}
		The symbol class $\mathcal{P}_e^0$ is defined in Section 2. 
		\begin{proof}
		Let $\mathbb{B}_r=\{x\in\Rn:|x|<r\}$. Let $\{(U_j,B_j,\kappa_j)\}_{j=1}^K$ be a collection of open sets on $M$. Each $(U_j,\kappa_j)$ is a local coordinate chart,  $\kappa_j(B_j) = \mathbb{B}_1$ and $\mathbb{B}_3\subset \kappa_j(U_j)$. We may assume that there is a partition of unity $\{\phi^0_j\}_{j=1}^K$ adapted to $\{B_j\}_{j=1}^K$. Let $\phi\in C_0^{\infty}(\Rn)$ so that $\phi(x) = 1$ when $|x|<2$ and $\phi(x) = 0$ when $|x|>3$, denote $\phi^1_j=\kappa_j^{-1}(\mathbb{B}_3)$. For each $j$, we construct a parametrix $P_j(t)$ in the form of \eqref{paraformula} in $U_j$, then
			$$P_t = \sum_{j=1}^{K} \phi^1_jP_j(t)\phi^0_j$$
			is a parametrix for the Poisson integral operator $\mathcal{H}$ near the boundary.
			So we will only work in local coordinates on $M$. As in \cite[Section 12 of Chapter 7]{TaylorPDE2},  we only need to construct a parametrix $P_t$ for the following evolution equation
			\[\partial_t u + A(t)u = 0,\ u(0)=f,\]
	since $P_t$ also gives a parametrix for the Poisson integral operator $\mathcal{H}$ by \eqref{lapdec}.
			
			We build the symbol of $P_t$ by the series $b(t,x,\xi)\sim \sum_{j\ge0}b_j(t,x,\xi)$, which should satisfy the equation
			\[\partial_t b+\sum_{\alpha}\frac1{\alpha!}D_\xi^\alpha a\partial_x^\alpha b=0,\ \ b(0,x,\xi)=1.\]
			To determine each term of the series, we require that
			\begin{equation}\label{b0}\partial_t b_0+ab_0=0,\ \ b(0,x,\xi)=1,\end{equation}
			and for $j\ge1$,
			\begin{equation}\label{bj}\partial_t b_j+\sum_{\ell=0}^j\sum_{|\alpha|=j-\ell}\frac1{\alpha!}D_\xi^\alpha a\partial_x^\alpha b_\ell=0,\ \ b_j(0,x,\xi)=0.\end{equation}
			Solving \eqref{b0} gives
			\begin{equation}\label{b0eq}
				b_0(t,x,\xi)=e^{-\int_0^ta(s,x,\xi)ds}\in \mathcal{P}_e^0,
			\end{equation}
			since $a(s)$ is a smooth family in $S^1$.
			To solve \eqref{bj}, we let 
			\[r_j(t,x,\xi)=-\sum_{\ell=0}^{j-1}\sum_{|\alpha|=j-\ell}\frac1{\alpha!}D_\xi^\alpha a\partial_x^\alpha b_\ell.\]
			Then \eqref{bj} is equivalent to
			\[\partial_t b_j+ab_j=r_j,\ \ b_j(0,x,\xi)=0\]
			which gives the recursive formula
			\begin{align*}
				b_j(t,x,\xi)&=b_0(t,x,\xi)\int_0^t r_j(s,x,\xi)b_0(s,x,\xi)^{-1}ds\\
				&=-b_0(t,x,\xi)\sum_{\ell=0}^{j-1}\sum_{|\alpha|=j-\ell}\frac1{\alpha!}\int_0^t D_\xi^\alpha a(s,x,\xi)\partial_x^\alpha  b_\ell(s,x,\xi) \cdot b_0(s,x,\xi)^{-1}ds.
			\end{align*}
			We claim that for $j\ge1$ 
			\begin{equation}\label{bjeq}
				b_j(t,x,\xi)=b_0(t,x,\xi)\sum_{2\le k\le 2j}t^kq_{j,k}(t,x,\xi),
			\end{equation}
			where $q_{j,k}(t)$ are smooth families in $S^{k-j}$. Then we have  $b_j\in \mathcal{P}_e^{-j}$ by \eqref{b0eq}. So we can build a symbol $b\sim \sum b_j\in \mathcal{P}_e^0$ (see e.g.  \cite[Lemma 3.1.2]{SoggeFIO}). Now we prove the claim \eqref{bjeq} by induction. For $j=1$, 
			\begin{align*}
				b_1(t,x,\xi)&=-b_0(t,x,\xi)\sum_{|\alpha|=1}\int_0^tD_\xi^\alpha a(s,x,\xi)\partial_x^\alpha b_0(s,x,\xi) \cdot b_0(s,x,\xi)^{-1}ds\\
				&=-b_0(t,x,\xi)\sum_{|\alpha|=1}\int_0^tD_\xi^\alpha a(s,x,\xi)\int_0^s\partial_x^\alpha a(\tau,x,\xi)d\tau ds\\
				&=b_0(t,x,\xi)t^2q_{1,2}(t,x,\xi)
			\end{align*}
			where $q_{1,2}(t)$ is a smooth family in $S^1$. Suppose \eqref{bjeq} holds for $b_\ell$ with $1\le \ell<j$. Then for $|\alpha|\ge1$
			\begin{align*}
				\partial_x^\alpha b_0(t,x,\xi)=b_0(t,x,\xi)\sum_{m=1}^{|\alpha|}t^mq_m(t,x,\xi)
			\end{align*}
			and for $1\le \ell<j$ and $|\alpha|\ge0$
			\begin{align*}
				\partial_x^\alpha b_\ell(t,x,\xi)&=\sum_{2\le k\le 2\ell}\sum_{|\alpha_1|+|\alpha_2|=|\alpha|}t^k\partial_x^{\alpha_1}b_0(t,x,\xi)\partial_x^{\alpha_2}q_{\ell,k}(t,x,\xi)\\
				&=b_0(t,x,\xi)\sum_{2\le k\le 2\ell +|\alpha|}t^k\tilde q_{k-\ell}(t,x,\xi)
			\end{align*}
			where $q_m(t),\ \tilde q_m(t)$ are smooth families in $S^m$. 
			Thus, we have
			\begin{align*}
				b_j(t,x,\xi)&=-b_0(t,x,\xi)\sum_{\ell=0}^{j-1}\sum_{|\alpha|=j-\ell}\frac1{\alpha!}\int_0^t D_\xi^\alpha a(s,x,\xi)\partial_x^\alpha  b_\ell(s,x,\xi) \cdot b_0(s,x,\xi)^{-1}ds\\
				&=-b_0(t,x,\xi)\sum_{|\alpha|=j}\sum_{m=1}^{j}\frac1{\alpha!}\int_0^t D_\xi^\alpha a(s,x,\xi) s^mq_{m}(s,x,\xi)ds\\
				&\quad\quad-b_0(t,x,\xi)\sum_{\ell=1}^{j-1}\sum_{|\alpha|=j-\ell}\sum_{2\le k\le j+\ell}\frac1{\alpha!}\int_0^t D_\xi^\alpha a(s,x,\xi) s^k\tilde q_{k-\ell}(s,x,\xi)ds\\
				&=b_0(t,x,\xi)\sum_{2\le k\le 2j}t^kq_{j,k}(t,x,\xi).
			\end{align*}
			where $q_{j,k}(t)$ are smooth families in $S^{k-j}$. So the claim \eqref{bjeq} is proved.
		\end{proof}
\subsection{Kernel of the composition operator}
		
		In the following, let $P\in OPS^1(M)$ be classical and self-adjoint with principal symbol $p(x,\xi)=|\xi|_{g(x)}$.  Let $\chi\in \mathcal{S}(\mathbb{R})$ satisfy $\supp \hat \chi\subset (\frac12,1)$ and $\chi(0)=1$. By \cite[Lemma 5.1.3]{SoggeFIO}, we can write the kernel of the spectral cluster as
	\[\chi(P-\la)(x,y)=\la^{\frac{n-1}2}e^{i\la d_g(x,y)}a_\la(x,y)+r_\la(x,y) \]
	where $a_\la(x,y)$ is supported in $\{d_g(x,y)\approx 1\}$ and $|\partial_{x,y}^\alpha a_\la(x,y)|\le C_\alpha,\ \forall \alpha.$
	And the remainder term is smooth and satisfies
	\[|\partial_{x,y}^\alpha r_\la(x,y)|\le C_\alpha\la^{-N},\ \ \forall N,\alpha.\]
 Recall $g_t$ is the metric on $\Sigma_t$ under Fermi coordinates and $g=g_0$. Let 
	\begin{equation}\label{rg}r(t)=\inf_{x\in M,\ \xi\ne0}\frac{|\xi|_{g_t(x)}}{|\xi|_{g(x)}},\ \quad G(t)=\int_0^tr(s)ds.\end{equation}
	We have $r(t)=1+O(t)$ and $G(t)=t+O(t^2)$.
	One can check that $G(t)=K(t)$ if $\Omega$ is a Euclidean ball  or the warped product manifold in Proposition \ref{symbound}. Now we can calculate the kernel of $\mathcal{H}\circ\chi(P-\la)$, i.e. the composition of the Poisson integral operator with the spectral cluster,  by applying the parametrix in Lemma \ref{parametrix}.
	\begin{lemma}\label{intkernel}
	Fix $0< \delta_0<\delta$. For any $0\le t\le \delta_0$, we have	\begin{equation}
			\mathcal{H}(\chi(P-\la)f)(t,x)=\la^{\frac{n-1}2}e^{-c\la G(t)}\int e^{i\la d_g(x,y)}a_{\la}(t,x,y)f(y)dz+  \int r_{\la}(t,x,y) f(y)dy,
		\end{equation}
		where $a_{\la}(t,x,y)$ is supported in $\{d_g(x,y)\approx 1\}$ and $|\partial_t^k\partial_{x,y}^\alpha a_{\la}(t,x,y)|\le C_{\alpha,k}\la^k,\ \forall \alpha,k.$
		And the remainder term has a  smooth kernel satisfying
		\[|\partial_t^k\partial_{x,y}^\alpha r_{\la}(t,x,y)|\le C_{\alpha,k}\la^{-N},\ \ \forall N,\alpha,k.\]
		The positive constants $c, C_{\alpha,k}$ are independent of $\la$.
	\end{lemma}
	\begin{proof}
	For convenience, we choose $x$ to be geodesic normal coordinates at $x_0$ and work in a small neighborhood of $x_0$.
	Let $\chi_\la=\chi(P-\la)-r_\la$. By Lemma \ref{parametrix} we can write
	\begin{align*}
		\mathcal{H}(\chi(P-\la)f)(t,x)&=P_t\chi(P-\la)f(x)+R_t\chi(P-\la)f(x)\\
		&=P_t\chi_\la f(x) +P_t r_\la f(x) + R_t\chi(P-\la)f(x).
	\end{align*}
	Here
	\begin{align*}
		P_t \chi_\la f(x)&=(2\pi)^{-n}\iint e^{i(x-y)\cdot \xi}b(t,x,\xi)\chi_\la f(y)dyd\xi\\
		&=(2\pi)^{-n}\la^{\frac{n-1}2}\iiint e^{i(x-y)\cdot \xi+i\la d_g(y,z)}b(t,x,\xi) a_\la(y,z)f(z)dzdyd\xi.
	\end{align*}
	Since $P_t$ is a smooth family of operators in $OPS^0$, it is smoothing off diagonal, and the contribution of $y$ away from $x$ is $O(\la^{-N})$. Thus, we can assume the integral with respect to $y$ in the expression of $P_t$ is taking near diagonal and $\max\{d_g(x,x_0),d_g(y,x)\}\ll d_g(x_0,z)$.
	Since the kernel of $R_t$ is smooth, we can integrate by parts to see that the kernel of $R_t\chi(P-\la)$ is smooth and $O(\la^{-N})$.  Similarly, the kernel of $P_tr_\la$ is also smooth and $O(\la^{-N})$.
	Let $\beta\in C_0^\infty(\mathbb{R})$ satisfy
	\[\1_{[\frac12,2]}\le \beta\le \1_{[\frac14,4]}.\] Let
	\begin{align*}
		R_t^\la f(x)&=(2\pi)^{-n}\la^{\frac{n-1}2}\iiint e^{i(x-y)\cdot \xi+i\la d_g(y,z)}b(t,x,\xi)(1-\beta(|\xi|_{g(x)}/\la)) a_\la(y,z)f(z)dzdyd\xi.
	\end{align*}
	 On the support of $1-\beta(|\xi|_{g(x)}/\la)$, and $g_{ij}(x)\approx g_{ij}(x_0)=\delta_{ij}$, $d_g(y,z)\approx|y-z|$, we have
	\[|\partial_y((x-y)\cdot\xi+\la d_g(y,z))|\gs \la+|\xi|.\]
	So we integrate by parts in $y$ to see the kernel of $R_t^\la$ is smooth and $O(\la^{-N})$.
	Let
	\begin{align*}
		P_t^\la f(x)&=(2\pi)^{-n}\la^{\frac{n-1}2+n}\iiint e^{i\la((x-y)\cdot \xi+ d_g(y,z))}b(t,x,\la\xi)\beta(|\xi|_{g(x)}) a_\la(y,z) f(z)dyd\xi dz.
	\end{align*}
	By the support property of $a_\la(y,z)$, we can integrate by parts in $\xi$ to see that the kernel
	\begin{align*}
		\Big|\iint e^{i\la((x-y)\cdot \xi+ d_g(y,z))}(1-\eta(x,z))b(t,x,\la\xi)\beta(|\xi|_{g(x)})a_\la(y,z)dyd\xi\Big|\ls \la^{-N}
	\end{align*}
	if $\eta\in C^\infty$ equals 1 when $|x-z|\approx 1$ and 0 otherwise. 	Since $a(t,x,\xi)=|\xi|_{g_t(x)}\mod S^0$,  by \eqref{rg} on the support of $\beta(|\xi|_{g(x)})$ we have 
	$$\beta(|\xi|_{g(x)})\ge \la r(t)|\xi|_{g(x)}+O(1)\gs \la r(t).$$ Then for some $c>0$ we can write 
	\begin{equation}\label{bexp0}
		b(t,x,\la\xi)=e^{-c\la G(t)}\tilde b(t,x,\la\xi),
	\end{equation}where on the support of $\beta(|\xi|_{g(x)})$ we have
	\[|\partial_t^k\partial_x^\beta\partial_\xi^\alpha \tilde b(t,x,\la \xi)|\le C_{\alpha,\beta,k}\la^{k},\ \forall \alpha,\beta,k.\]
	So it suffices to handle the operator
	\begin{align*}\tilde P_t^\la f(x)&=(2\pi)^{-n}\la^{\frac{n-1}2+n}\iiint e^{i\la((x-y)\cdot \xi+ d_g(y,z))}\eta(x,z)b(t,x,\la\xi)\beta(|\xi|_{g(x)}) a_\la(y,z) f(z)dyd\xi dz\\
		&=(2\pi)^{-n}\la^{\frac{n-1}2+n}e^{-c\la G(t)}\iiint e^{i\la((x-y)\cdot \xi+ d_g(y,z))}\eta(x,z)\tilde b(t,x,\la\xi)\beta(|\xi|_{g(x)}) a_\la(y,z) f(z)dyd\xi dz.
	\end{align*}
	We shall use stationary phase in $(y,\xi)$. Indeed,
	\[\partial_{y,\xi}((x-y)\cdot\xi+ d_g(y,z))=0\Rightarrow y=x,\ \xi=\partial_yd_g(y,z),\]
	and
	\[\partial_{y,\xi}^2((x-y)\cdot\xi+ d_g(y,z))=\begin{bmatrix}
		\partial_y^2 d_g(y,z) & -I\\
		-I & 0
	\end{bmatrix}.
	\]
	Thus, $(y,\xi)=(x,\partial_yd_g(y,z))$ is a non-degenerate critical point of the phase function. By stationary phase \cite[Corollary 1.1.8]{SoggeFIO}, we can write
	\begin{align*}
		\tilde P_t^\la f(x)=\la^{\frac{n-1}2}e^{-c\la G(t)}\int e^{i\la d_g(x,z)} a_{\la}(t,x,z)f(z)dz,
	\end{align*}
	where $a_{\la}(t,x,z)$ is supported in $\{d_g(x,z)\approx 1\}$ and $|\partial_t^k\partial_{x,y}^\alpha a_{\la}(t,x,y)|\le C_{\alpha,k}\la^k,\ \forall \alpha,k.$\end{proof}
	
By using an essentially the same argument, we can obtain the kernel of the composition operator $Q\circ\chi(P-\la)$ for any $Q\in OPS^m(M)$. 
	\begin{lemma}\label{intkernel2}
		Let $m\in \mathbb{R}$ and $Q\in OPS^m(M)$. We have	
		\begin{equation}
		(Q\circ\chi(P-\la))(x,y)=\la^{\frac{n-1}2+m} e^{i\la d_g(x,y)}\tilde a_{\la}(x,y)+ \tilde r_{\la}(x,y)
		\end{equation}
		where $\tilde a_{\la}(x,y)$ is supported in $\{d_g(x,y)\approx 1\}$ and $|\partial_{x,y}^\alpha \tilde a_{\la}(x,y)|\le C_{\alpha},\ \forall \alpha.$
		And the remainder term is  smooth and satisfies
		\[|\partial_{x,y}^\alpha \tilde r_{\la}(x,y)|\le C_{\alpha},\ \ \forall\alpha.\]
		The positive constants $C_{\alpha}$ are independent of $\la$.
	\end{lemma}
	This lemma implies that the kernel of $Q\circ \chi(P-\la)$   has essentially the same form as the kernel of $\chi(P-\la)$. So one may use it to obtain natural eigenfunction estimates under the operations in $OPS^m$, which extend the seminal works by Sogge \cite{SoggeFIO} and Burq-G\'erard-Tzvetkov \cite{BGT, BGTbi,BGTbi2}.
	\begin{prop}\label{Qsoggelp}
		Let $m\in \mathbb{R}$ and $Q\in OPS^m(M)$. For $\la\ge1$,  we have
		\begin{equation}\label{Qlp}	\|Q\circ\chi(P-\la)f\|_{L^p(M)}\ls \la^{m+\sigma(p) }\|f\|_{L^2(M)},\ \ 2\le p\le\infty,\end{equation}
		where\[
		\sigma(p)=\begin{cases}\frac{n-1}2(\frac12-\frac1p),\ \ \  2\le p<\frac{2(n+1)}{n-1}\\
			\frac{n-1}2-\frac np,\ \ \ \ \ \  \frac{2(n+1)}{n-1}\le p\le \infty.\end{cases}\]	
	\end{prop}
	
		\begin{prop}\label{Qbgtbi}(Bilinear estimates)
		Let $m_j\in \mathbb{R}$ and $Q_j\in OPS^{m_j}(M)$ for $j=1,2$. If $\mu\ge\la\ge1$,  then we have
		\begin{equation}\label{Qbi}	\|Q_1\circ\chi(P-\la)f\cdot Q_2\circ\chi(P-\mu)g \|_{L^2(M)}\ls \la^{m_1}\mu^{m_2}B\|f\|_{L^2(M)}\|g\|_{L^2(M)},\end{equation}
		where\[
		B=\begin{cases}
		\la^\frac14,\ \ \ \ \ \ \ \ \ \ n=2\\
		\la^{\frac12}\sqrt{\log\la},\ \ n=3\\
		\la^{\frac{n-2}2},\ \ \ \ \ \ \  n\ge4.\end{cases}\]	
	\end{prop}
		\begin{prop}\label{Qbgtlp}(Restriction estimates)
		Let $m\in \mathbb{R}$ and $Q\in OPS^m(M)$. Let $\Sigma$ be a $k$-dimensional submanifold in $M$. For $\la\ge1$,  we have
		\begin{equation}\label{Qbgt1}	\|Q\circ\chi(P-\la)f\|_{L^p(\Sigma)}\ls \la^{m}A\|f\|_{L^2(M)},\ \ 2\le p\le\infty,\end{equation}
		where
		\begin{itemize}
			\item If $k\le n-3$, then $A= \la^{\frac{n-1}2-\frac{k}p}$ for $2\le p\le \infty$.
			\item If $k=n-2$, then $A= \la^{\frac{n-1}2-\frac{k}p}$ for $2< p\le \infty$, and $A= \la^{\frac12}(\log\la)^\frac12$ for $p=2$.
			\item If $k=n-1$, then $A= \la^{\frac{n-1}2-\frac{k}p}$ for $\frac{2n}{n-1}\le p\le \infty$, and $A= \la^{\frac{n-1}4-\frac{n-2}{2p}}$ for $2\le p<\frac{2n}{n-1}$.
		\end{itemize}
	\end{prop}
	Let $\1_{(\la-1,\la]}$ be the indicator function of  $(\la-1,\la]$. Since we may assume $\chi$ is non-vanishing on $[-1,1]$, we have $\chi(P-\la)\tilde f=\1_{(\la-1,\la]}(P)f$ if $\tilde f=\chi(P-\la)^{-1}\1_{(\la-1,\la]}(P)f$. Therefore,  the propositions above still hold if we replace $\chi(P-\la)$ by $\1_{(\la-1,\la]}(P)$.

		\section{Interior decay and comparable norms}
		In this section, we prove Theorem \ref{thm2} and Theorem \ref{thm3}, by using microlocal analysis techniques involving the parametrix of the Poisson integral operator in Lemma \ref{parametrix}. Note that the symbol of the parametrix $P_t$  has an exponential decay factor $\sim e^{-ct\langle \xi\rangle}$, which formally gives the desired exponential decaying property if the frequency $|\xi| \gs \la$. This is basic idea in the proof.

		Let $P\in OPS^1$ be classical and self-adjoint with principal symbol $p(x,\xi)$ positive on $T^*M\setminus 0$. Let $\mu\in \mathbb{R}$ and let $m\in C^\infty(\mathbb{R})$ satisfy
		\begin{equation}
			|m^{(k)}(\tau)|\le C_k (1+|\tau|)^{\mu-k},\ k=0,1,2,...
		\end{equation} Let $\la\ge1$ and $m_\la(\tau)=m(\tau/\la)$. We have $m_\la\in S^{\mu_1}$ if $\mu_1=\max\{\mu,0\}$, since
\begin{equation}
	|m_\la^{(k)}(\tau)|\le C_k (1+|\tau|)^{\mu_1-k}, \ k=0,1,2,...
\end{equation}
		The constant $C_k$ is independent of $\la$. We need the following lemma on the kernel of $m_\la(P)$, which can be viewed as a more precise version of \cite[Theorem 4.3.1]{SoggeFIO}. We shall prove it in the appendix.
		
		\begin{lemma}\label{pdolemma}
		 Let $\mu,\ \mu_1,\ m,\ m_\la$ defined as above. We have
			\[m_\la(P)(x,y)=(2\pi)^{-n}\int e^{i(x-y)\cdot \xi}M_\la(x,\xi)d\xi+R_\la(x,y)\]	
			where the symbol $M_\la\in S^{\mu_1}$ satisfies
			\begin{equation}\label{mlasym}
				M_\la(x,\xi)\sim \sum_{\ell=0}^\infty m_\la^{(\ell)}(p(x,\xi)) q_\ell(x,\xi),\ \  q_\ell\in S^{\lfloor \ell/2\rfloor},
			\end{equation}
			and the remainder term $R_\la$ is smooth with
			\[|\partial_{x,y}^\alpha R_\la(x,y)|\le C_\alpha,\ \ \forall \alpha.\]
			The symbols $ q_\ell$ and the constant $C_\alpha$ are independent of $\la$.
		\end{lemma}
			In the following, we assume  $P\in OPS^1(M)$ is classical and self-adjoint with principal symbol $p(x,\xi)=|\xi|_{g(x)}$. 
		Let $\1_{\ge \la}$ be the indicator function of $[\la,\infty)$, and let $f_{\ge\la}=\1_{\ge \la}(P)f$.

	\noindent \textbf{Proof of Theorem \ref{thm2}.}
	Since $(\mathcal{H}f)(t,x)=P_tf(x)+R_tf(x)$ with $R_t\in OPS^{-\infty}$ by Lemma \ref{parametrix}, we shall estimate the norms $\|P_tf_{\ge\la}\|_{L^p(M)}$ and $\|R_tf_{\ge\la}\|_{L^p(M)}$.

	If $\{e_j\}$ is an orthonormal eigenbasis associated with the eigenvalues $\{\la_j\}$ of the operator $P$,  then by the H\"ormander's $L^2-L^\infty$ bounds \eqref{soggelp} we have
	\[\sum_{\la_j\ge\la}(1+\la_j)^{-N-n}|e_j(x)|^2\ls \sum_{k\ge \la}k^{-N-1}\ls \la^{-N},\ \forall N>0.\]
	Since $R_t$ is smoothing,
	\begin{align*}
		\|R_tf_{\ge\la}\|_{L^\infty(M)}&=\|R_t(1+P)^{N+n}(1+P)^{-N-n}f_{\ge\la}\|_{L^\infty(M)}\\
		&\ls \|(1+P)^{-N-n}f_{\ge\la}\|_{L^\infty(M)}\\
		&\ls \la^{-N}\|f_{\ge\la}\|_{L^1(M)},\ \ \forall N>0.
	\end{align*}
Now we begin to handle $P_t$. By Lemma \ref{parametrix}, we have \[P_t f(x)=(2\pi)^{-n}\int e^{ix\cdot \xi}b(t,x,\xi)\hat f(\xi)d\xi\]
	where  the symbol
	\[b(t,x,\xi)\sim b_0(t,x,\xi)+\sum_{j\ge1}\sum_{2\le k\le 2j}b_0(t,x,\xi)t^k q_{j,k}(t,x,\xi)\]
	with $q_{j,k}(t)$ are smooth families in $S^{k-j}$ and $b_0(t,x,\xi)=e^{-\int_0^t a(s,x,\xi)ds}\in \mathcal{P}_e^0$.
	Let $0<\eps <\frac{1}{100}$ and $\eta\in C^\infty(\mathbb{R})$ satisfy
	\[\1_{[1,\infty)}\le \eta\le 1_{[1-\eps,\infty)}.\]
	Let $\eta_1\in C^\infty(\mathbb{R})$ satisfy
	\[\1_{[1-\eps,\infty)}\le \eta_1\le 1_{[1-2\eps,\infty)}.\]
	Clearly, we have $\eta(P/\la)f_{\ge\la}=f_{\ge\la}$.
	We split
	\begin{align*}
		P_tf_{\ge\la}(x)&=\int e^{ix\cdot \xi}b(t,x,\xi)\eta_1(p(x,\xi)/\la)(f_{\ge\la})^\wedge(\xi)d\xi+\int e^{ix\cdot \xi}b(t,x,\xi)(1-\eta_1(p(x,\xi)/\la))(\eta(P/\la)f_{\ge\la})^\wedge(\xi)d\xi\\
		&:=\tilde P_tf_{\ge\la}(x)+\tilde R_tf_{\ge\la}(x).
	\end{align*}
	The operator $\tilde R_t$
	is smoothing by Lemma \ref{pdolemma}, since the supports of $1-\eta_1$ and $\eta$ are disjoint. Then we also have
\begin{equation}
	\|\tilde R_t f_{\ge\la}\|_{L^\infty(M)}\ls \la^{-N}\|f_{\ge\la}\|_{L^1(M)},\ \forall N.
\end{equation}
	Next, we handle $\tilde P_t$. 
	The kernel of $\tilde P_t$ is 
	\[K_t(x,y)=(2\pi)^{-n}\int e^{i(x-y)\cdot \xi}b(t,x,\xi)\eta_1(p(x,\xi)/\la)d\xi.\]
	On the support of $\eta_1(p(x,\xi)/\la)$ we have $p(x,\xi)=|\xi|_{g(x)}\ge (1-2\eps)\la$. Since the principal symbol of $a(t,x,\xi)$ is $|\xi|_{g_t(x)}$, by \eqref{rg} we have 
	$$|\xi|_{g_t(x)} \ge r(t)|\xi|_{g(x)}\ge (1-2\eps)\la r(t).$$
	Recall \eqref{rg}. Then we can write
	\begin{equation}\label{bexp}
		b(t,x,\xi)=e^{-(1-3\eps)\la G(t)}\tilde b(t,x,\xi),
	\end{equation}where
	\[|\partial_t^k\partial_x^\beta\partial_\xi^\alpha \tilde b(t,x,\xi)|\le C_{\alpha,\beta,k}(1+|\xi|)^{k-|\alpha|},\ \forall \alpha,\beta,k.\]
	
	Since $\tilde b(t,x,\xi)\eta_1(p(x,\xi)/\la)\in S^0$, we have for $1<p<\infty$
\begin{equation}
	\|\tilde P_t\|_{L^p(M)\to L^p(M)}\ls e^{-(1-3\eps)\la G(t)}.
\end{equation}
	For $p=\infty$ or 1, we need to calculate the kernel. We have
\begin{equation}\label{ker1}
	|K_t(x,y)|\ls e^{-(1-3\eps)\la G(t)}\int_{|\xi|\ge\la}e^{-\eps t|\xi|/2}d\xi\ls e^{-(1-3\eps)\la G(t)}t^{-n}.
\end{equation}
	But integration by parts gives
\begin{equation}\label{ker2}
	|K_t(x,y)|\ls e^{-(1-3\eps)\la G(t)} |x-y|^{-N}\int_{|\xi|\ge\la}|\xi|^{-N}\ls e^{-(1-3\eps)\la G(t)} \la^n(\la|x-y|)^{-N}.
\end{equation}
	So by splitting the integral into two parts: $|x-y|\le \la^{-1}$ and $|x-y|>\la^{-1}$ we have
	\[\int|K_t(x,y)|dx\ls e^{-(1-3\eps)\la G(t)}(1+(\la t)^{-n}).\]
	By Young's inequality, for $p=\infty$ or 1
\begin{equation}
	\|\tilde P_t\|_{L^p(M)\to L^p(M)}\ls e^{-(1-3\eps)\la G(t)}(1+(\la t)^{-n}).
\end{equation}
	Then for $p=\infty$ or 1
	\begin{equation}\label{Linfty}
		\|\mathcal{H}f_{\ge\la}\|_{L^p(\Sigma_t)}\ls e^{-(1-3\eps)\la G(t)}(1+(\la t)^{-n})\|f_{\ge\la}\|_{L^p(M)}+\la^{-N}\|f_{\ge\la}\|_{L^1(M)}.
	\end{equation}
	But by the maximum principle, when $t\le\la^{-1}$, $$\|\mathcal{H}f_{\ge\la}\|_{L^\infty(\Sigma_t)}\le \|f_{\ge\la}\|_{L^\infty(M)}.$$
	So we always have for $1<p\le \infty$
	\begin{equation}\label{Lpexp}
		\|\mathcal{H}f_{\ge\la}\|_{L^p(\Sigma_t)}\ls e^{-(1-3\eps)\la G(t)}\|f_{\ge\la}\|_{L^p(M)}+\la^{-N}\|f_{\ge\la}\|_{L^1(M)}.
	\end{equation}
	Furthermore, if $f$ has spectral frequency $\approx\la$, then \eqref{Lpexp} is valid for all $1\le p\le\infty$ by repeating the argument above, since \eqref{ker1} can be improved in this case.
	So we finish the proof of Theorem \ref{thm2}.
		
	Next, we show that if $f$ has spectral frequency $\approx\la$, than  $\mathcal{H}f$ also has a lower bound near the boundary.
	
		\begin{lemma}\label{lplowbounds}
			Let $\beta\in C_0^\infty(\mathbb{R}^+)$ and $f_\la=\beta(P/\la)f$.  For $t \leq \lambda^{-1}$, we have for any $1 \leq p \leq \infty$,
		\begin{equation}\label{lowbd}
			\|\mathcal{H}f_\la\|_{L^p(\Sigma_t)}\gs \|f_\la\|_{L^p(M)}.
		\end{equation}
		\end{lemma}
\begin{proof}
 We may assume $\supp \beta\subset (\frac12,1)$.		Since $(\mathcal{H}f)(t,x)=P_tf(x)+R_tf(x)$ with $R_t\in OPS^{-\infty}$ by Lemma \ref{parametrix}, we shall estimate the norms $\|P_tf_{\la}\|_{L^p(M)}$ and $\|R_tf_{\la}\|_{L^p(M)}$. As before, we have
	\[\|R_tf_\la\|_{L^\infty(M)}\ls \la^{-N}\|f_\la\|_{L^1(M)},\ \forall N.\]
	Here we use the fact that $f_\la$ has spectral frequency bounded below.
	
	To handle $P_t$, we split $b(t,x,\xi)=b_0(t,x,\xi)+r(t,x,\xi)$ with $r(t,x,\xi)\in S^{-1}$. Let
	\begin{align*}
		r_tf(x)=(2\pi)^{-n}\int e^{ix\cdot \xi}r(t,x,\xi)\hat f(\xi)d\xi.
	\end{align*}
	Let $\beta_1\in C_0^\infty$ satisfy
	\[\1_{[\frac12,1]}\le \beta_1\le \1_{[\frac14,2]}.\]
	Clearly, $\beta_1\beta=\beta$.
	Since $t\le \la^{-1}$, we can obtain the kernel estimate
	\begin{align*}
		|r_t\beta_1(P/\la)(x,y)|\ls \la^{n-1}(1+\la|x-y|)^{-N},\ \forall N.
	\end{align*}
	Then for $1\le p\le \infty$ by Young's inequality we have
	\begin{align*}
		\|r_tf_\la\|_{L^p(M)}&=\|r_t\beta_1(P/\la)f_\la\|_{L^p(M)}\ls \la^{-1}\|f_\la\|_{L^p(M)}.
	\end{align*}
	Next, let
	$\beta_2\in C_0^\infty$ satisfy
	\[\1_{[\frac18,4]}\le \beta_2\le \1_{[\frac1{16},8]}.\]
	We split $b_0=\beta_2(|\xi|_{g(x)}/\la))b_0+(1-\beta_2(|\xi|_{g(x)}/\la)))b_0$. Let
	\begin{align*}
		\tilde R_tf(x)=(2\pi)^{-n}\int e^{ix\cdot \xi}(1-\beta_2(|\xi|_{g(x)}/\la))b_0(t,x,\xi)\hat f(\xi)d\xi.
	\end{align*}
	Since $\supp (1-\beta_2)$ is disjoint with $\supp\beta_1$, the operator $\tilde R_t \beta_1(P/\la)$ is smoothing by Lemma \ref{pdolemma}, and then
	\[\|\tilde R_t f_\la\|_{L^\infty(M)}=\|\tilde R_t \beta_1(P/\la)f_\la\|_{L^\infty(M)}\ls \la^{-N}\|f_\la\|_{L^1(M)}.\]
	Let 
\begin{equation}
	P_t^\la f(x)=(2\pi)^{-n}\int e^{ix\cdot \xi}b_0(t,x,\xi)\beta_2(|\xi|_{g(x)}/\la)\hat f(\xi)d\xi.
\end{equation}
	Notice that $b_0(t,x,\xi)$ behaves like an elliptic symbol in $S^0$ near $|\xi|_{g(x)} \sim \lambda$ when $t\le \la^{-1}$, one can construct a microlocal inverse of $P_t^\la$ as
\begin{equation}
	E_t^\la f(x)=(2\pi)^{-n}\int e^{ix\cdot \xi}b_0(t,x,\xi)^{-1}\beta_2(|\xi|_{g(x)}/\la)\hat f(\xi)d\xi.
\end{equation}
	Since $t\le \la^{-1}$, we have the kernel estimate
	\begin{align*}
		|E_t^\la(x,y)|\ls \la^n(1+\la|x-y|)^{-N}.
	\end{align*}
	Then by Young's inequality for $1\le p\le \infty$
	\[\|E_t^\la\|_{L^p(M)\to L^p(M)}\ls 1.\]
	Then we can write
	\begin{align*}
		E_t^\la P_t^\la f(x)
		&=f(x)+(2\pi)^{-n}\int e^{ix\cdot \xi}(1-\beta_2(|\xi|_{g(x)}/\la)^2)\hat f(\xi)d\xi+(2\pi)^{-n}\int e^{ix\cdot \xi}q_\la(t,x,\xi)\hat f(\xi)d\xi\\
		&:=f(x)+R_t^\la f(x)+Q_t^\la f(x)
	\end{align*}
	where the symbol $q_\la\in S^{-1}$ with $|\xi|\approx\la$. Then integration by parts gives  the kernel estimate
	\[|Q_t^\la(x,y)|\ls \la^{n-1}(1+\la|x-y|)^{-N},\ \forall N.\]
	By Young's inequality, this implies
\begin{equation}
	\|Q_t^\la\|_{L^p(M)\to L^p(M)}\ls \la^{-1}.
\end{equation}
	Since $R_t^\la\beta_1(P/\la)$ is smoothing, we have
	\[\|R_t^\la f_\la\|_{L^\infty(M)}=\|R_t^\la\beta_1(P/\la)f_\la\|_{L^\infty(M)}\ls \la^{-N}\|f_\la\|_{L^1(M)}.\]
	Thus, 
	\begin{align*}
		\|P_t^\la f_\la\|_{L^p(M)}&\gs \|E_t^\la P_t^\la f_\la\|_{L^p(M)}\\
		&\ge \|f_\la\|_{L^p(M)}+O(\la^{-1})\|f_\la\|_{L^p(M)}
	\end{align*}
	Combing the estimates above, we obtain \eqref{lowbd} for $\la$ larger than a constant $\la_0$. For $\la<\la_0$, we can apply the exponential lower bound in Theorem \ref{thm1} and the maximum principle, since $$\|f_\la\|_{L^1(M)}\approx \|f_\la\|_{L^2(M)}\approx \|f_\la\|_{L^{\infty}(M)}$$ in this case. 
	\end{proof}

	\noindent \textbf{Proof of Theorem \ref{thm3}.} Take $\delta_0 = \delta/2$, we split the $L^p$ norm into two parts:
			\[\|\mathcal{H}f_\la\|^p_{L^p(\Omega)} = \|\mathcal{H}f_\la\|^p_{L^p(\Omega_{\delta_0})} + \|\mathcal{H}f_\la\|^p_{L^p(\Omega_{\delta_0}^c)}.\]
		For the lower bound, we only need to consider the contribution of  the boundary part. If $\la \ge \delta_0^{-1}$, then by Lemma \ref{lplowbounds} and the co-area formula we have
		\begin{equation}\label{lowglobal}
			\begin{aligned}
				\int_{\Omega_{\delta_0}^c} |\mathcal{H}f_\la|^p &= \int_{0}^{\delta_0} \int_{\Sigma_t} |\mathcal{H}f_\la|^p d\omega_tdt\\
				&\gs \int_{0}^{\la^{-1}}dt \|f_\la\|_{L^p(M)}^p\\
				&\gs \la^{-1}\|f_\la\|_{L^p(M)}^p.
			\end{aligned}
		\end{equation}
		The case that $\la<\delta_0^{-1}$ can be handled similarly since $\delta_0$ is a constant. For the upper bound,  the interior part follows from Theorem \ref{thm2} and the maximum principle, since $$\|\mathcal{H}f_\la\|_{L^\infty(\Omega_{\delta_0})} \ls \la^{-N}\|f_\la\|_{L^1(M)}, \ \forall N.$$ The boundary part follows from Theorem \ref{thm2} and the co-area formula, since for $1\le p\le \infty$
		\[\|\mathcal{H}f_\la\|_{L^p(\Sigma_t)}\ls (e^{-c\la t}+\la^{-N})\|f_\la\|_{L^p(M)},\ \forall N.\]

	\section{Applications to Steklov eigenfunctions and numerical approximation}

		Over the past two decades, efficient spectral methods  have been developed in the field of computational  mathematics, for the approximation of solutions to Laplace’s equations using Steklov eigenfunctions. See e.g., Auchmuty \cite{A04,A06,A17,A18,A182}, Auchmuty-Cho \cite{AC15,AC17}, Auchmuty-Rivas \cite{AR16}, Cho \cite{Cho20}, Cho-Rivas \cite{CR22}, Klouček-Sorensen-Wightman \cite{KSW07}. Among the key problems in this area are the error estimates for approximations and the orthogonality properties of the Steklov eigenfunctions. In this section, we shall use the results established in the preceding sections to address these issues.
		
		We first prove the almost-orthogonality in Theorem \ref{thm4}.
	
		%\noindent \textbf{}
	\noindent \textbf{Proof of Theorem \ref{thm4}.}
			Suppose $\mu\ge\la$. If $\mu\le 10$ or $|\la-\mu|\le 10$, \eqref{inp} just follows from Cauchy-Schwarz, Lemma \ref{intkernel} and Theorem \ref{thm2}. In the case of $\mu\ge 10$ and $\la \le 10$, since the derivatives of the kernel of $\chi(P-\la)$ are bounded by constants independent with $\la$, we have 
			$$\langle v_{\la}, w_{\mu} \rangle_\Omega=O(\mu^{-N})\|f\|_{L^2(M)}\|g\|_{L^2(M)},\ \forall N.$$ 
			Now suppose $\la\ge 10$ and $|\la-\mu|>10$. Let $\gamma=\la/\mu\in (0,1)$. As before, we only need to handle the contribution near boundary. It suffices to show that on each slice $\Sigma_t$
		\begin{equation}
			|\langle v_{\la}(t), w_{\mu}(t) \rangle_{\Sigma_t}| \ls e^{-ct(\la+\mu)}(1+|\la-\mu|)^{-N}\|f\|_{L^2(M)}\|g\|_{L^2(M)},\ \ \forall N,
		\end{equation}
			and then integration with respect to $t$ gives the desired bound.
			By Lemma \ref{intkernel} and Theorem \ref{thm2}, it suffices to handle the kernel
			\[B_{\la,\mu}(y,z)=(\la\mu)^{\frac{n-1}2}\int e^{i\mu(\gamma d_g(x,y)- d_g(x,z))}a_\la(x,y)\overline{a_\mu(x,z)}dx\]	since the contributions from other terms are $O(\mu^{-N})$.
			By Cauchy-Schwarz, we only  need to prove	\begin{equation}\label{inner}
				\iint |B_{\la,\mu}(y,z)|^2dydz\ls (1+|\la-\mu|)^{-N},\ \forall N,
			\end{equation}
			Fix any $\eps\in (0,\frac1{100})$. Let $\nabla$ be the covariant derivative with respect to the metric $g$. Then we have $|\nabla_x d(x, y)|_{g(x)}=1$. For convenience we take geodesic normal coordinates at $x$. If $\gamma< 1-\eps$, then
			\[|\nabla_x(\gamma d_g(x,y)- d_g(x,z))|\ge \eps. \]Then integration by parts gives $B_{\la,\mu}(y,z)=O(\mu^{-N})$.

			Suppose $1-\eps\le \gamma \le 1$. If $|\nabla_xd_g(x,y)-\nabla_x d_g(x,z)|\ge 2\eps$, then 
			\[|\nabla_x(\gamma d_g(x,y)- d_g(x,z))|\ge |\nabla_x( d_g(x,y)- d_g(x,z))|-(1-\gamma)\ge \eps.\]
			Then integration by parts gives $B_{\la,\mu}(y,z)=O(\mu^{-N})$.
			Suppose  $|\nabla_xd_g(x,y)-\nabla_x d_g(x,z)|<2\eps$. We may assume that these two unit vectors are close to $e_1=(1,0,...,0)$. Then $$|\nabla_xd_g(x,y)-\nabla_x d_g(x,z)|\gs |y'-z'|$$ if $y=(y_1,y')$ and $z=(z_1,z')$. So we have 
			$$|\nabla_x(\gamma d_g(x,y)- d_g(x,z))| \gs  (1-\gamma) + |y'-z'|.$$
			Then integration by parts gives
			$$|B_{\la,\mu}(y,z)|\ls \mu^{n-1}  (1+|\la-\mu|)^{-N}(1+\mu|y'-z'|)^{-N},\ \forall N.$$
			Then a direct calculation gives \eqref{inner}.

	\noindent \textbf{Proof Sketch of Theorems \ref{thm5} and \ref{thm6}.} The proofs of Theorems \ref{thm5} and \ref{thm6} are similar, so we only give a proof sketch. We can reduce the estimates in $\Omega$ to the bounds on each slice $\Sigma_t$, which follow from Lemma \ref{intkernel} and the arguments in \cite{BGT} and \cite{BGTbi, BGTbi2}. The power gains in Theorems \ref{thm5} and \ref{thm6} come from the integration of the exponential decay factor in the kernel formula in Lemma \ref{intkernel}.

		 Now we shall discuss the numerical approximation problems of harmonic functions with various boundary conditions. Klouček-Sorensen-Wightman \cite[Section 4.3]{KSW07} and Auchmuty-Cho \cite[Sections 4 and 6]{AC17} have obtained some basic error estimates and provided accompanying numerics. As an application of our results, we provide both $L^2$ and pointwise error estimates. For $x\in \Omega$, let $$d(x) = \dist(x,\partial \Omega)=\dist(x,M).$$
		 Let $\{e_j\}_{j=0}^\infty$ be the  orthonormal DtN eigenbasis in $L^2(M)$ associated with the Steklov eigenvalues $$0=\la_0< \la_1\le \la_2\le ...,$$
		 which are arranged in increasing order and we account for multiplicity. The first eigenvalue $\la_0=0$ is simple since $\Omega$ is connected.
		 Let $u_j=\mathcal{H}e_j$ be the associated Steklov eigenfunctions.

		\begin{theorem}\label{errorboundsD}
		Let $u$ be the solution to Laplace equation \eqref{lapdirichlet} with Dirichlet boundary condition. 	Let $k\in \mathbb{N}_+$, $\la = \la_{k+1}$, and $c_j = \langle f,e_{j} \rangle_M$. Let $\tilde f_k = \sum_{j \le k} c_j e_{j}$ and $\tilde u_k = \sum_{j\le k} c_j u_{j}$. Then we have the $L^2$ error estimates
			\begin{equation}\label{2to2}
				\|u-\tilde u_k\|^2_{L^2(\Omega)} \ls \la^{-1}\Big(\|f\|^2_{L^2(M)}-\|\tilde f_k\|^2_{L^2(M)}\Big )
			\end{equation}
			and the pointwise error estimates
			\begin{equation}\label{2toinf}
				|u(x)-\tilde u_k(x)|^2 \ls \Big((\la+ d(x)^{-1})^{n}e^{-\la d(x)}+\la^{-N}\Big)\Big(\|f\|^2_{L^2(M)}-\|\tilde f_k\|^2_{L^2(M)}\Big),\ \ \forall  N.
			\end{equation}
		\end{theorem}
		The bound \eqref{2to2} immediately follows from \eqref{intLpup}, since
		\[\|f\|^2_{L^2(M)}-\|\tilde f_k\|^2_{L^2(M)}=\sum_{\la_j\ge\la}|c_j|^2=\|f-\tilde f_k\|_{L^2(M)}^2.\] To get \eqref{2toinf}, we use Theorem \ref{thm2} and H\"ormander's  $L^\infty$ bound \eqref{soggelp}. Indeed, for $0<c<\frac34$ we have
		\begin{align*}
			|u(x)-\tilde u_k(x)|&\le \sum_{2^\ell\ge \la}|\sum_{\la_j\approx 2^\ell} c_j u_j|\\
			&\ls \sum_{2^\ell\ge \la}(e^{-cd(x) 2^\ell }2^{\frac{n}2\ell}+2^{-N\ell})(\sum_{\la_j\approx 2^\ell}|c_j|^2)^\frac12\\
			&\ls \Big(\sum_{2^\ell\ge \la}(e^{-2cd(x) 2^\ell }2^{n\ell}+2^{-2N\ell})\Big)^\frac12(\sum_{\la_j\ge\la}|c_j|^2)^\frac12.
		\end{align*}
		Then a direct calculation gives \eqref{2toinf}.
		%The $L^2$ error estimates in \eqref{2to2} is sharp by . 
		
		%For the pointwise estimate \eqref{2toinf}. Let $\Omega$ be the unit ball in $\mathbb{R}^{n+1}$. By checking the Zonal functions with frequency $\sim d(x)^{-1}$, one can see that the term $d(x)^{-\frac{n-1}{2}}$ in the \eqref{2toinf} is also sharp.
		
		Next, let $b\ge0$ be a fixed number and we consider the Laplace equation with Neumann ($b=0$) or Robin ($b>0$) boundary condition
		\begin{equation}\label{lapNR}
			\begin{cases}
				\Delta u(x)=0,\ x\in\Omega,\\
				\partial_\nu u(x) + bu(x) = f(x),\ x\in \partial \Omega=M.
			\end{cases}
		\end{equation}
		\begin{theorem}\label{errorboundsNR}
		Let $u$ be the solution to Laplace equation \eqref{lapNR} with Neumann or Robin boundary condition.  Let $k\in \mathbb{N}_+$, $\la = \la_{k+1}$, and $c_j = \langle f,e_{j} \rangle_M$. Let $\tilde f_k = \sum_{ j \le k} c_j e_{j}.$ When $b=0$, let  $$\tilde u_k = \sum_{1\le j\le k} \frac{c_j}{\la_j} u_{j}.$$ When $b>0$, let $$  \tilde u_k = \sum_{j\le k} \frac{c_j}{\la_j+b} u_{j}.$$  Then we have the $L^2$ error estimates
		\begin{equation}\label{2to2NR}
			\|u-\tilde u_k\|^2_{L^2(\Omega)} \ls \la^{-3}\Big(\|f\|^2_{L^2(M)}-\|\tilde f_k\|^2_{L^2(M)}\Big ).
		\end{equation}
		and the pointwise error estimates
		\begin{equation}\label{2toinfNR}
			|u(x)-\tilde u_k(x)|^2 \ls \Big((\la+d(x)^{-1})^{n-2}e^{-\la d(x)}+\la^{-N}\Big)\Big(\|f\|^2_{L^2(M)}-\|\tilde f_k\|^2_{L^2(M)}\Big), \ \ \forall N.
		\end{equation}
	\end{theorem}	
		These estimates immediately follow from the same proof as Theorem \ref{errorboundsD} and the observation that
\begin{equation}
	c_j=\langle f,e_{j} \rangle_M=\langle \partial_\nu u+b u,e_{j} \rangle_M=(\la_j+b)\langle u,e_{j} \rangle_M.
\end{equation}
	
	%	\section{Further Discussions}
		
	%	Provide some topics and open problems.
		
	%	\begin{problem}
	%		Upper bound for smooth manifolds?
	%	\end{problem}
	%	If we want to adapt the proof of Theorem \ref{thm1} to prove some kind of exponential upper bound, the key is to show an estimates like
	%	$$\frac{\int_{\Sigma_t} |\partial_t u|^2 \int_{\Sigma_t} u^2 - (\int_{\Sigma_t} u\partial_t u)^2}{(\int_{\Sigma_t} u^2)^2}\ls F(t) \frac{\int_{\Sigma_t} %u\partial_t u}{\int_{\Sigma_t} u^2},$$
	%	However
		
	%	\begin{problem}results for Calderon problem,for example distinguishablity or instability?
	%	\end{problem}

		\section{Appendix: Proof of Lemma \ref{pdolemma}}
		We prove Lemma \ref{pdolemma} in the appendix. We essentially follow the strategy in the proof of \cite[Theorem 4.3.1]{SoggeFIO}. But Lemma \ref{pdolemma} is more precise than \cite[Theorem 4.3.1]{SoggeFIO}, and we need to carefully handle the dependence on $\la$. So we shall give a detailed proof for completeness.
		
		Let $\mu_1=\max\{\mu,0\}$.	We fix $\rho\in \mathcal{S}(\mathbb{R})$ with $\hat \rho$ supported in $\{|t|<1\}$ and equal to 1 in $\{|t|<\frac12\}$. 
		We decompose
		\begin{align*}
			m_\la(P)&=(2\pi)^{-1}\int \hat\rho(t)\hat m_\la(t)e^{itP}dt+(2\pi)^{-1}\int (1-\hat\rho(t))\hat m_\la(t)e^{itP}dt\\
			&:=\tilde m_\la(P)+r_\la(P).
		\end{align*}
		Here $r_\la(P)$  is a smoothing operator, since
		\begin{align*}
			|r_\la(\tau)|&=|(2\pi)^{-1}\int (1-\hat\rho(t))\hat m_\la(t)e^{it\tau}dt|\\
			&=|(2\pi)^{-1}\iint (1-\hat\rho(t))m_\la(s)e^{-its}e^{it\tau}dsdt|\\
			&=C_N|\iint (1-\hat\rho(t)) t^{-N}m_\la^{(N)}(s)e^{-its}e^{it\tau}dsdt|\\
			&\ls\int (1+|s-\tau|)^{-N_1}|m_\la^{(N)}(s)|ds\\
			&\ls \int (1+|s-\tau|)^{-N_1}(1+|s|)^{\mu_1-N}ds\\
			&\ls (1+\tau)^{\mu_1-N}
		\end{align*}
		and its derivatives also satisfy similar bounds.
		
		Without loss of generality, we assume the parametrix for half-wave kernel (\cite[Theorem 4.1.2]{SoggeFIO}) holds for $|t|\le 1$, that is
		\[e^{itP}(x,y)= Q(t,x,y)+R(t,x,y)\]
		where
		\[ Q(t,x,y)=(2\pi)^{-n}\int e^{i\varphi(x,y,\xi)+itp(y,\xi)}q(t,x,y,\xi)d\xi\]and $R(t,x,y)\in C^\infty([-1,1]\times M\times M)$ and $Q(t,x,y)$ is supported in a small neighborhood of the diagonal in $M\times M$. Moreover, $\varphi\in S^1$ and satisfies
		\[\varphi(x,y,\xi)=(x-y)\cdot\xi +O(|x-y|^2|\xi|)\]and
		\[q(0,x,x,\xi)-1\in S^{-1}.\]
		The integral of $R(t)$ is a smoothing operator since
		\begin{align*}
			\Big|\int \hat\rho(t)\hat m_\la(t)R(t,x,y)dt\Big|&=\Big|\iint  m_\la(s)\hat\rho(t)R(t)e^{-its}dtds\Big|\\
			&\ls \int |m(s/\la)|(1+|s|)^{-N}ds\\
			&\ls 1
		\end{align*}and its derivatives also satisfy similar bounds.
		
		Next, we consider the integral of $Q(t,x,y)$.
		By Taylor expansion, we have
		\[q(t)=q(0)+tq'(0)+...+\frac{t^N}{N!}q^{(N)}(0)+\frac{t^{N+1}}{N!}\int_0^1q^{(N+1)}(\theta t)(1-\theta)^Nd\theta.\]
		
		For the remainder term, it contributes the kernel
		\begin{equation}\label{err}
			\int_0^1\iint \hat\rho(t)\hat m_\la(t)e^{i\varphi(x,y,\xi)+itp(y,\xi)}t^{N+1}q^{(N+1)}(\theta t,x,y,\xi) d\xi dt(1-\theta)^Nd\theta.
		\end{equation}
		Let
		\begin{align*}
			E_{\la,\theta}(x,y,\xi)=	\int \hat\rho(t)\hat m_\la(t)e^{itp(y,\xi)}t^{N+1}q^{(N+1)}(\theta t,x,y,\xi) dt
		\end{align*}
		Since $p(y,\xi)\approx |\xi|$ and $(\hat \rho q^{(N+1)}(\theta\cdot))^\vee\in\mathcal{S}$, we have
		\begin{align*}
			|E_{\la,\theta}(x,y,\xi)|&\approx  \Big|\int \hat\rho(t) (m_\la^{(N+1)})^\wedge(t)e^{itp(y,\xi)}q^{(N+1)}(\theta t,x,y,\xi) dt\Big|\\
			&\approx | m_\la^{(N+1)}*(\hat \rho q^{(N+1)}(\theta\cdot))^\vee(p(y,\xi))|\\
			&\ls \int \la^{-\mu}(\la+|\tau|)^{\mu-N-1}(1+|p(y,\xi)-\tau|)^{-N_1}d\tau\\
			&\ls (1+|\xi|)^{\mu_1-N-1}
		\end{align*}
		Similarly, 	\begin{align*}
			|\partial_{x,y}^\beta \partial_\xi^\alpha E_{\la,\theta}(x,y,\xi)|\ls (1+|\xi|)^{\mu_1-N-1-|\alpha|},\ \ \forall \alpha.
		\end{align*}
		Thus, \eqref{err} is the kernel of an operator in $OPS^{\mu_1-N-1}$ by \cite[Theorem 3.2.1]{SoggeFIO}.
		
		For $\ell=0,1,2,...N,$ we have
		\begin{align*}
			&\iint \hat\rho(t)\hat m_\la(t)e^{i\varphi(x,y,\xi)+itp(y,\xi)}t^\ell q^{(\ell)}(0,x,y,\xi)d\xi dt\\
			&=\int \tilde m_\la^{(\ell)}(p(y,\xi))q^{(\ell)}(0,x,y,\xi)e^{i\varphi(x,y,\xi)}d\xi\\
			&=\int  m_\la^{(\ell)}(p(y,\xi))q^{(\ell)}(0,x,y,\xi)e^{i\varphi(x,y,\xi)}d\xi+R_{\ell,\la}(x,y)
		\end{align*}
		where 
		\begin{align*}
			|R_{\ell,\la}(x,y)|&=\Big|\int r_\la^{(\ell)}(p(y,\xi))q^{(\ell)}(0,x,y,\xi)e^{i\varphi(x,y,\xi)}d\xi\Big|\\
			&\ls \int (1+|\xi|)^{-N}d\xi\\
			&\ls 1
		\end{align*}
		since $m_\la-\tilde m_\la=r_\la\in \mathcal{S}$. Its  derivatives also satisfy similar bounds.
		
		Let $\varphi_0=\varphi$ and $\varphi_1=(x-y)\cdot\xi$ and $\varphi_t=(1-t)\varphi_0+t\varphi_1$ with $t\in [0,1]$. Let $P(x,y,\xi)=m_\la^{(\ell)}(p(y,\xi))q^{(\ell)}(0,x,y,\xi)$. Let
		\[(P_t u)(x)=(2\pi)^{-n}\iint e^{i\varphi_t(x,y,\xi)}P(x,y,\xi)u(y)d\xi dy.\]
		Since  \begin{equation}\label{phicond}|\partial_\xi^\alpha(\varphi_0-\varphi_1)|\ls |x-y|^2|\xi|^{1-|\alpha|},\ \forall\alpha,\end{equation}
		we have
		\[|\nabla_\xi \varphi_t|\gs |x-y|\]
		near the diagonal. 
		
		By \eqref{phicond}, we can define  $\phi(x,y,\xi)=(\varphi_1-\varphi_0)/|x-y|^2\in S^1$. So 
		\begin{align*}
			\partial_t^jP_tu&=(2\pi)^{-n}\iint e^{i\varphi_t}(i(\varphi_1-\varphi_0))^jP(x,y,\xi)u(y)d\xi dy
		\end{align*}
		and then
		\begin{align*}
			\partial_t^j|_{t=1}P_tu&=(2\pi)^{-n}\iint e^{i(x-y)\cdot\xi}(i(\varphi_1-\varphi_0))^jP(x,y,\xi)u(y)d\xi dy\\
			&=(2\pi)^{-n}(-1)^j\iint e^{i(x-y)\cdot\xi}\Delta_\xi^j\Big((i\phi(x,y,\xi))^jP(x,y,\xi)\Big)u(y)d\xi dy\\
			&:=(2\pi)^{-n}(-1)^j\iint e^{i(x-y)\cdot\xi}Q_j(x,y,\xi)u(y)d\xi dy
		\end{align*}
		By Taylor formula we have
		\[Q_j(x,y,\xi)=\sum_{|\alpha|\le N}\frac1{\alpha!}\partial_y^\alpha|_{y=x} Q(x,y,\xi)\cdot (y-x)^\alpha+R_N(x,y,\xi)\]
		where $R_N$ vanishes of order $N+1$ when $x=y$. For the remainder term, we use the reproducing operator (see \cite[Lemma 4.1.1]{hangzhou}) $$L e^{i\phi}=e^{i\phi},\ \ \text{if}\ L=\frac1{i}\frac{\nabla\phi}{|\nabla\phi|^2}\cdot\nabla$$ to rewrite the kernel
		\begin{align*}
			\int e^{i(x-y)\cdot\xi}R_N(x,y,\xi)d\xi&=\int e^{i(x-y)\cdot\xi}(L^*)^{N+1}R_N(x,y,\xi)d\xi
		\end{align*}
		Since that the adjoint of $L$ is $$L^*f=i\frac{(x-y)\cdot \nabla_\xi f}{|x-y|^2}$$
		we have 
		\[|(L^*)^{N+1}R_N(x,y,\xi)|\ls (1+|\xi|)^{\mu_1-\ell -N-1}\]
		Similarly,
		\[|\partial_{x,y}^\beta\partial_\xi^\alpha(L^*)^{N+1}R_N(x,y,\xi)|\ls (1+|\xi|)^{\mu_1-\ell -N-1-|\alpha|},\ \ \forall \alpha,\beta.\]
		So the remainder term is an operator in $OPS^{\mu_1-\ell -N-1}$.
		Thus, by \cite[(3.1.11)]{SoggeFIO} the operator  $\partial_t^j|_{t=1}P_t$ is a pseudo-differential operator with symbol 
		\[Q_j(x,\xi)\sim (-1)^j\sum_\alpha \frac1{\alpha!}D_\xi^\alpha\partial_y^\alpha|_{y=x} Q_j(x,y,\xi)\]
		modulo a smoothing operator. Then by Taylor formula
		\[P_0=\sum_{j=0}^N\frac{(-1)^j}{j!}\partial_t^j|_{t=1}P_t+\frac{(-1)^{N-1}}{N!}\int_0^1t^N\partial_t^{N+1}P_t dt\]
		For the remainder term, we use the reproducing operator $L e^{i\varphi_t}=e^{i\varphi_t}$ and  rewrite the kernel 
		\begin{align*}
			\int e^{i\varphi_t}(i(\varphi_1-\varphi_0))^{N+1}P(x,y,\xi)d\xi&=\int e^{i\varphi_t}(L^*)^{2N+2}((i(\varphi_1-\varphi_0))^{N+1}P(x,y,\xi))d\xi\\
			&:=\int e^{i\varphi_t}E_{N}(x,y,\xi)d\xi
		\end{align*}
		Since the adjoint of $L$ is $$L^*f=i\sum_j\partial_{\xi_j}(\frac{\partial_{\xi_j}\varphi_t}{|\nabla_\xi\varphi_t|^2}\cdot f)$$
		we have the symbol
		\begin{align*}
			|E_{N}(x,y,\xi)|\ls (1+|\xi|)^{\mu_1-\ell-N-1}.
		\end{align*}
		Similarly,
		\[	|\partial_{x,y}^\beta\partial_\xi^\alpha E_{N}(x,y,\xi)|\ls (1+|\xi|)^{\mu_1-\ell -N-1-|\alpha|},\ \ \forall \alpha,\beta.\]
		So the remainder term is an operator in $OPS^{\mu_1-\ell -N-1}$ by \cite[Theorem 3.2.1]{SoggeFIO}. Thus, $P_0$ is a pseudo-differential operator with symbol given by the formal series
		\[Q(x,\xi)\sim \sum_{j=0}^\infty\frac{1}{j!}Q_j(x,\xi)\]
		modulo a smoothing operator. 
		
		Thus, $m_\la(P)$ is a peeudo-differential operator with symbol given by the formal series
		\[M_\la(x,\xi)\sim \sum_{\ell=0}^\infty\sum_{j=0}^\infty\sum_\alpha\frac1{\ell!j!\alpha!}D_\xi^\alpha\partial_y^\alpha|_{y=x}\Delta_\xi^j\Big((i(\phi(x,y,\xi))^jm_\la^{(\ell)}(p(y,\xi))q^{(\ell)}(0,x,y,\xi)\Big)\]
		modulo a smoothing operator. Here $q^{(\ell)}(0)\in S^0$ and  $$\phi(x,y,\xi)=(\varphi_1-\varphi_0)/|x-y|^2\in S^1.$$
		We can further simplify the formal series. For convenience, we use $q_k$ to denote some symbol in $S^{k}$ and it may vary from line to line. Then
		\begin{align*}
			M_\la(x,\xi)&\sim \sum_{\ell=0}^\infty\sum_{j=0}^\infty\sum_\alpha \sum_{0\le \alpha_1\le \alpha}\sum_{0\le \alpha_2\le \alpha+2j}m_\la^{(\ell+\alpha_1+\alpha_2)}(p(x,\xi))q_{\alpha_1+j-(\alpha+2j-\alpha_2)}(x,\xi)\\
			&=\sum_{\ell=0}^\infty\sum_{j=0}^\infty\sum_\alpha \sum_{0\le k\le 2\alpha+2j}m_\la^{(\ell+k)}(p(x,\xi))q_{k-\alpha-j}(x,\xi)\\
			&=\sum_{\ell=0}^\infty\sum_{j=0}^\infty\sum_\alpha \sum_{0\le k\le \min(\ell,2\alpha+2j)}m_\la^{(\ell)}(p(x,\xi))q_{k-\alpha-j}(x,\xi)\\
			&=\sum_{\ell=0}^\infty m_\la^{(\ell)}(p(x,\xi))\tilde q_{\ell}(x,\xi)
		\end{align*}
		where $\tilde q_\ell\in S^{\lfloor\ell/2\rfloor}$.
		The leading term only comes from the term with $\ell=j=\alpha=0$, so it is exactly
		\[m_\la(p(x,\xi))q(0,x,x,\xi)\]
		which is $m_\la(p(x,\xi)) \mod S^{\mu_1-1}$.
		\bibliographystyle{plain}

\begin{thebibliography}{0}
				\bibitem{A79} Almgren, Frederick J., Jr.
				Dirichlet's problem for multiple valued functions and the regularity of mass minimizing integral currents.Minimal submanifolds and geodesics (Proc. Japan-United States Sem., Tokyo, 1977), pp. 1–6
				North-Holland Publishing Co., Amsterdam-New York, 1979
				
			\bibitem{A04}Auchmuty, Giles. Steklov eigenproblems and the representation of solutions of elliptic boundary
			value problems, Numer. Funct. Anal. Optim., 25 (2004), 321-348.
			
				\bibitem{A06}Auchmuty, Giles.
			Spectral characterization of the trace spaces $H^s(\partial \Omega)$.
			SIAM J. Math. Anal.   38 (2006), no. 3, 894–905.
			
			\bibitem{A17}Auchmuty, Giles.
			The S.V.D. of the Poisson kernel.
			J. Fourier Anal. Appl.   23 (2017), no. 6, 1517–1536.
			
			\bibitem{A18}Auchmuty, Giles.
			Steklov representations of Green's functions for Laplacian boundary value problems.
			Appl. Math. Optim.   77 (2018), no. 1, 173–195.
			\bibitem{A182}Auchmuty, Giles.
			Robin approximation of Dirichlet boundary value problems.
			Numer. Funct. Anal. Optim.   39 (2018), no. 10, 999–1010.
			
			\bibitem{AC15}Auchmuty, Giles; Cho, Manki.
			Boundary integrals and approximations of harmonic functions.
			Numer. Funct. Anal. Optim.   36 (2015), no. 6, 687–703.
			\bibitem{AC17}Auchmuty, Giles; Cho, Manki.
			Steklov approximations of harmonic boundary value problems on planar regions.
			J. Comput. Appl. Math.   321 (2017), 302–313.
			
			
					\bibitem{AR16}Auchmuty, Giles; Rivas, M. A.
					Laplacian eigenproblems on product regions and tensor products of Sobolev spaces.
					J. Math. Anal. Appl.   435 (2016), no. 1, 842–859.
				
					\bibitem{M71}Boutet de Monvel, Louis. Boundary problems for pseudo-differential operators. Acta Math. 126 (1971), no. 1-2, 11–51.
			\bibitem{BGTbi}Burq N, Gérard P, Tzvetkov N. Bilinear eigenfunction estimates and the nonlinear Schrödinger equation on surfaces. Inventiones mathematicae. 2005 Jan;159(1):187-223.
			
			\bibitem{BGTbi2}Burq N, Gérard P, Tzvetkov N. Multilinear estimates for the Laplace spectral projectors on compact manifolds. Comptes Rendus Mathematique. 2004 Mar 1;338(5):359-64.
			\bibitem{BGT}Burq, N., G\'erard, P. and Tzvetkov, N., Restrictions of the Laplace-Beltrami eigenfunctions to submanifolds, Duke Math. J. 138 (2007), 445–486.
			\bibitem{Cho20}Cho, Manki.
			Steklov expansion method for regularized harmonic boundary value problems.
			Numer. Funct. Anal. Optim.   41 (2020), no. 15, 1871–1886.
				\bibitem{CR22}Cho, Manki and Rivas, Mauricio A. On the $L^2$-orthogonality of Steklov eigenfunctions. Proceedings of the 2021 UNC Greensboro PDE Conference, 45–58.
				Electron. J. Differ. Equ. Conf., 26
				\bibitem{CGGS24}Colbois, Bruno; Girouard, Alexandre; Gordon, Carolyn; Sher, David.
				Some recent developments on the Steklov eigenvalue problem.
				Rev. Mat. Complut.   37 (2024), no. 1, 1–161.
			\bibitem{DHN21}Daudé, T., Helffer, B., Nicoleau, F.: Exponential localization of Steklov eigenfunctions on warped
			product manifolds: the flea on the elephant phenomenon. Ann. Math, Québec (2021)
			
			\bibitem{DR2019}Di Cristo, M., and L. Rondi. Interior decay of solutions to elliptic equations with respect to frequencies at the boundary. Indiana University Mathematics Journal 70.4 (2021): 1303-1334.
				
			\bibitem{GT2019}Galkowski, J., Toth, J.A.. Pointwise bounds for Steklov eigenfunctions. J. Geom. Anal. 29(1), 142–193 (2019)
			\bibitem{GT2021}Galkowski, J., Toth, J.A..	Lower bounds for Steklov eigenfunctions. Pure Appl. Math. Q.   19 (2023), no. 4, 1873–1898.
			\bibitem{GL1986} Garofalo, Nicola, and Lin, Fang-Hua. Monotonicity properties of variational integrals, $A_p$ weights and unique continuation. Indiana University Mathematics Journal 35, no. 2 (1986): 245-268.
			\bibitem{GP17}Girouard, A., Polterovich, I.: Spectral geometry of the Steklov problem (survey article). J. Spectr.
			Theory 7(2), 321–359 (2017)
			\bibitem{HK22}Helffer, B., Kachmar, A.: Semi-classical edge states for the robin laplacian. Mathematika 68(2), 454–485 (2022)
			\bibitem{HL2001}Hislop, Peter D., and Carl V. Lutzer. Spectral asymptotics of the Dirichlet-to-Neumann map on multiply connected domains in $\mathbb{R}^d$. Inverse Problems 17.6 (2001): 1717.
			\bibitem{hor68}Hörmander, Lars.
			The spectral function of an elliptic operator.
			Acta Math.   121 (1968), 193–218. 
			\bibitem{hor3}Hörmander, Lars. The Analysis of Linear Partial Differential Operators III Pseudo-Differential Operators, Springer-Verlag, Berlin, 1994.
					
				\bibitem{hu}R. Hu. $L^p$ norm estimates of eigenfunctions restricted to submanifolds. Forum Math., 6:1021–1052,
				2009.
			\bibitem{HSWZ2023}X. Huang, Y. Sire, X. Wang and C. Zhang. Sharp $L^p$ estimates and size of nodal sets for generalized Steklov eigenfunctions, arXiv preprint arXiv:2301.00095
			
			\bibitem{KSW07}Klouček, Petr; Sorensen, Danny C.; Wightman, Jennifer L. The approximation and computation of a basis of the trace space $H^{1/2}$. J Sci Comput. 2007; 32(1):73–108.
			
			\bibitem{LU89}Lee, John M.; Uhlmann, Gunther.
			Determining anisotropic real-analytic conductivities by boundary measurements.
			Comm. Pure Appl. Math.   42 (1989), no. 8, 1097–1112.
			\bibitem{LPPS22}Levitin, M., Parnovski, L., Polterovich, I., Sher, D.A.: Sloshing, Steklov and corners: asymptotics of
			Steklov eigenvalues for curvilinear polygons. Proc. Lond. Math. Soc. (3) 125(3), 359–487 (2022)
			
			\bibitem{Lin1991}Lin, Fang‐Hua. Nodal sets of solutions of elliptic and parabolic equations. Communications on Pure and Applied Mathematics 44, no. 3 (1991): 287-308.
				
			\bibitem{PST19}Polterovich, I., Sher, D.A., Toth, J.A.: Nodal length of Steklov eigenfunctions on real-analytic Riemannian
			surfaces. J. Reine Angew. Math. 754, 17–47 (2019)
			
			\bibitem{sogge88}Sogge, Christopher D.
			Concerning the $L^p$ norm of spectral clusters for second-order elliptic operators on compact manifolds.
			J. Funct. Anal.   77 (1988), no. 1, 123–138.
			\bibitem{SoggeFIO}Sogge, Christopher D. Fourier integrals in classical analysis. Vol. 210. Cambridge University Press, 2017.
			
			\bibitem{hangzhou}Sogge, Christopher D. Hangzhou lectures on eigenfunctions of the Laplacian, volume 188 of Annals of Mathematics
			Studies. Princeton University Press, Princeton, NJ, 2014.	
					
			\bibitem{tataru}Tataru, Daniel.
			On the regularity of boundary traces for the wave equation.
			Ann. Scuola Norm. Sup. Pisa Cl. Sci. (4)   26 (1998), no. 1, 185–206.
			
			\bibitem{TaylorPDE2}Taylor, Michael. Partial differential equations II: Qualitative studies of linear equations. Vol. 116. Springer Science and Business Media, 2013.
		
			
			%numeric references
		
			
		
			
	
		
		%	\bibitem{HK}Keller, Herbert B. Numerical methods for two-point boundary-value problems. Courier Dover Publications, 2018.
			
			 
		%	\bibitem{Sogge1985}Sogge, Christopher Donald. Oscillatory integrals and spherical harmonics. Princeton University, 1985.

		%	\bibitem{Martinez} Martinez, André. An introduction to semiclassical and microlocal analysis. Vol. 994. New York: Springer, 2002.
		%	\bibitem{Zworski} Zworski, Maciej. Semiclassical analysis. Vol. 138. American Mathematical Society, 2022.
		%	\bibitem{Sjo1996} Johannes Sj\"ostrand. Density of resonances for strictly convex analytic obstacles. Canad. J. Math., 48(2):397–447, 1996. With an appendix by M. Zworski.
		%	\bibitem{Martinez1994}Martinez, André. Estimates on complex interactions in phase space. Mathematische Nachrichten 167, no. 1 (1994): 203-254.
		%	\bibitem{Martinez1997}Martinez, A. Microlocal exponential estimates and applications to tunneling. Microlocal Analysis and Spectral Theory (1997): 349-376.
		%	\bibitem{Nakamura1995}Nakamura, Shu. On Martinez'method of phase space tunneling. Reviews in Mathematical Physics 7, no. 3 (1995): 431-442.
		%	\bibitem{Toth1998}Toth, John A. Eigenfunction decay estimates in the quantum integrable case. Duke Math. J. 95, no. 1 (1998): 231-255.
		%	\bibitem{GS1991}Guillemin, Victor, and Matthew Stenzel. Grauert tubes and the homogeneous Monge-Ampere equation. Journal of Differential Geometry 34, no. 2 (1991): 561-570.
		%	\bibitem{WF1959}Whitney, Hassler, and Bruhat, François. Quelques propriétés fondamentales des ensembles analytiques-réels. Commentarii Mathematici Helvetici 33, no. 1 (1959): 132-160.
		%	\bibitem{Grauert1958}Grauert, Hans. On Levi's problem and the imbedding of real-analytic manifolds. Annals of Mathematics (1958): 460-472.
		%	\bibitem{LU1989}Lee, John M., and Uhlmann, Gunther. Determining anisotropic real‐analytic conductivities by boundary measurements. Communications on Pure and Applied Mathematics 42, no. 8 (1989): 1097-1112.
		%	\bibitem{PS2015}Polterovich, Iosif, and Sher, David A.  Heat invariants of the Steklov problem. The Journal of Geometric Analysis 25 (2015): 924-950.
		%	\bibitem{SU2016}Sjöstrand, Johannes, and Uhlmann, Gunther. Local analytic regularity in the linearized Calderón problem. Analysis and PDE 9, no. 3 (2016): 515-544.
		
			
		
		
		
			
		
			
		
			
			
		\end{thebibliography}
		
	\end{document}